\documentclass{article}

\usepackage{arxiv}

\usepackage[utf8]{inputenc} 
\usepackage[T1]{fontenc}    
\usepackage{url}            
\usepackage{booktabs}       
\usepackage{amsfonts}       
\usepackage{nicefrac}       
\usepackage{microtype}      
\usepackage{lipsum}

\usepackage[unicode]{hyperref}

\usepackage{color}
\usepackage{latexsym}
\usepackage{indentfirst}
\usepackage{amsxtra}
\usepackage{amssymb}
\usepackage{amsthm}
\usepackage{amsmath}
\usepackage{mathrsfs} 

\usepackage[makeroom]{cancel}
 
\usepackage{soul}
\usepackage{tikz}

\usepackage[capitalise]{cleveref}

\newtheorem{theor}{Theorem}
\newtheorem*{theor*}{Theorem}
\newtheorem{prop}[theor]{Proposition}
\newtheorem{lemma}[theor]{Lemma}
\newtheorem{cor}[theor]{Corollary}
\newtheorem*{cor*}{Corollary}
\theoremstyle{definition}               
\newtheorem{defin}[theor]{Definition}
\newtheorem{ex}{Example}
\newtheorem{exs}[ex]{Examples}
\newtheorem{rem}{Remark}

\newtheorem{que}[theor]{Question}

\DeclareMathOperator{\Sym}{Sym}
\DeclareMathOperator{\Aut}{Aut}
\DeclareMathOperator{\End}{End}
\DeclareMathOperator{\id}{id}

\DeclareMathOperator{\E}{E}

\DeclareMathOperator{\ind}{i}
\DeclareMathOperator{\per}{p}

\newcommand{\s}[1]{S_{#1}}
\newcommand{\phii}[2]{\phi_{#1,#2}}

\newcommand{\lambdaa}[2]{\lambda_{#1}{#2}}

\newcommand{\rhoo}[2]{\rho_{#1}{#2}}
\newcommand{\alphaa}[3]{\alpha^{#1}_{#2}{#3}}
\newcommand{\betaa}[3]{\beta^{#1}_{#2}{#3}}

\newcommand{\indd}[1]{\ind{#1}}
\newcommand{\perr}[1]{\per{#1}}

\title{Inverse semi-braces and the Yang-Baxter equation}

\author{
  Francesco~Catino
  \\
  Dipartimento di Matematica e Fisica\\ "Ennio De Giorgi"\\
  Università del Salento\\
    Via Provinciale Lecce-Arnesano\\
    73100 Lecce (Italy)\\
  \texttt{francesco.catino@unisalento.it} \\
   \And
Marzia ~Mazzotta \\
Dipartimento di Matematica e Fisica\\ "Ennio De Giorgi"\\
  Università del Salento\\
    Via Provinciale Lecce-Arnesano\\
    73100 Lecce (Italy)\\
  \texttt{marzia.mazzotta@unisalento.it} \\
  \And
 Paola ~Stefanelli \\
  Dipartimento di Matematica e Fisica\\ "Ennio De Giorgi"\\
  Università del Salento\\
    Via Provinciale Lecce-Arnesano\\
    73100 Lecce (Italy)\\
  \texttt{paola.stefanelli@unisalento.it} \\
}
\usepackage{setspace}

\begin{document}
\maketitle
\begin{abstract}
    The main aim of this paper is to provide set-theoretical
    solutions of the Yang-Baxter equation that are not necessarily bijective, among these new idempotent ones.
    In the specific, we draw on both to the classical theory of inverse semigroups and to that of the most recently studied braces, to give a new research perspective to the open problem of finding solutions. 
    Namely, we have recourse to a new structure, the \emph{inverse semi-brace}, that is a triple $(S,+, \cdot)$ with $(S,+)$ a semigroup and $(S, \cdot)$ an inverse semigroup satisfying the relation $a \left(b + c\right) = a  b + a\left(a^{-1} + c\right)$, for all $a,b,c \in S$, where $a^{-1}$ is the inverse of $a$ in $(S, \cdot)$.
    In particular, we give several constructions of inverse semi-braces which allow for obtaining solutions that are different from those until known.
\end{abstract}

\vspace{2mm}


\doublespacing

\section*{Introduction}
The quantum Yang-Baxter equation first appeared in theoretical physics in a paper by Yang \cite{Ya67} to study
a one-dimensional quantum mechanical many body problem. In an independent way, Baxter \cite{Ba72} solved an eight-vertex model in statistical mechanics by means of this equation. 
In subsequent years, the interest in such equation has vastly increased: it led to the foundations of the theory of quantum groups and it also appeared in
topology and algebra above all for its connections with braid groups and Hopf algebras. In the '90s, Drinfel'd \cite{Dr92} posed the question of finding all the so called set-theoretical solutions of the Yang-Baxter equation.
Specifically, given a set $S$, a map  $r: S \times S \longrightarrow S \times S$ satisfying the relation
\begin{align*}
    \left(r \times \id_S\right) \left(\id_S \times r\right)\left(r \times \id_S\right) 
    = \left(\id_S \times r\right)\left(r \times \id_S\right)\left(\id_S \times r\right)
\end{align*}
is said to be a \emph{set-theoretical solution} of the Yang-Baxter equation, or briefly a \emph{solution}. The map $r$ is usually written as $r(x,y) = (\lambda_x(y),\rho_y(x))$, with $\lambda_x$ and $\rho_y$ maps from $S$ into itself, for all $x,y \in S$. One says that a solution $r$ is \emph{left non-degenerate} if $\lambda_x$ is bijective, for every $x\in S$, \emph{right non-degenerate} if $\rho_y$ is bijective, for every $y\in S$, and \emph{non-degenerate} if $r$ is both left and right non-degenerate.
If $r$ is neither left nor right non-degenerate, then it is called \emph{degenerate}.
Determining all the solutions is still an open problem and it has drawn the attention of several mathematicians.   A large number of works related to this topic has been produced in recent years, actually. The milestones are the papers by Etingof, Schedler and Soloviev \cite{ESS99}, Gateva-Ivanova and Van den Bergh \cite{GaVa98}, Lu, Yan, and Zhu \cite{LuYZ00}, and Soloviev \cite{So00}, where a greater attention has been posed on non-degenerate bijective solutions. Subsequently, involutive solutions have been profusely investigated by many authors, as mostly illustrated into details in the introduction of the paper by Ced\'{o}, Jespers, and Okni\'{n}ski \cite{CeJeOk14}.
The most used approach is based on \emph{left braces}, algebraic structures introduced by Rump \cite{Ru07} that include the Jacobson radical rings. In particular, such structures are involved for obtaining non-degenerate solutions which are also \emph{involutive}, i.e., $r^2=\id$. 
In this way, Rump traced a novel research direction and later fruitful results on these kind of solutions appeared, as one can see in the survey by Ced\'{o} \cite{Ce18}, along the references therein.
To classify involutive solutions, Rump \cite{Ru05} also involved another algebraic structure, that is the \emph{left cycle set}. Interesting contributions in this framework have been obtained, for example, see \cite{Ru16, Ve16, CaCaPi18, CaCaMiPi19x, CaCaPi19, CaCaPi20, CaPiRu20,  Ru20-2}.\\
Bijective solutions, not necessarily involutive, can be produced through skew left braces, algebraic structures introduced by Guarnieri and Vendramin \cite{GuVe17}. Also in this case several works can be found, for instance \cite{SmVe18,CCoSt19,Ze19,CeSmVe19,JeKuVaVe19, Ru19-1,AcLuVe20,AcBo20, CaCaDC20}. Note that skew left  braces always produce solutions that are non-degenerate \cite[Theorem 3.1]{GuVe17} and, in the finite case, they are such that $r^{2n}= \id$, as shown by  Smoktunowicz and Vendramin \cite[Theorem 4.13]{SmVe18}. 
First instances of bijective solutions which are degenerate was found by Yang
\cite[Theorem 4.15]{Ya16}, who studied the interplay between $k$-graphs and the Yang-Baxter equation.

Recently, the focus has been gradually shifted to solutions that are not necessarily bijective. 
Among these solutions, the most studied are \emph{idempotent} ones. 
The investigation of such maps has been mainly started by Lebed \cite{Le17}, who provided a series of examples related to free and free commutative monoids, to factorizable monoids, and to distributive lattices. Later, Matsumoto and Shimizu \cite{MaSh18} also approached to idempotent solutions of dynamical type in a categorical framework. Moreover, Stanovsk{\`y} and Vojt{\v{e}}chovsk{\`y}  \cite{StVo20x} dealt with idempotent left non-degenerate  solutions which are in one-to-one correspondence with twisted Ward left quasigroups and, in particular, they enumerated those that are latin and idempotent. In addition, we mention Cvetko-Vah and Verwimp \cite{Charl19} who gave new examples of degenerate solutions including the idempotent ones, using the algebraic tool of the skew lattice. \\
More in general,  Catino, Colazzo, and Stefanelli \cite{CaCoSt17} showed that the algebraic structure of left semi-brace turns out to be a useful tool for producing left non-degenerate solutions which are not bijective.
Under mild assumptions, Jespers and Van Antwerpen \cite{JeAr19} determined soon after degenerate solutions through a slight generalization of left semi-braces.
Furthermore, construction techniques that allow for providing new solutions that are both not necessarily non-degenerate and bijective starting from given ones have been introduced. Specifically, in \cite{CaCoSt19} it is displayed how to find solutions of finite order by introducing the technique called the matched product of solutions inspired to the matched product of semi-braces \cite{CCoSt20}.
Another way to determine solutions of finite order which are non-bijective, even starting from bijective ones, is the strong semilattice of solutions, contained in \cite{CCoSt20x-2}, basing on the strong semilattice of semigroups. 
Concretely, instances of such solutions are obtained involving the structure of generalized semi-brace.
Catino, Mazzotta, and Stefanelli \cite{CaMaSt20} also provided new degenerate solutions
by means of a technique which involves solutions to the pentagon equation, another basic equation of mathematical
physics. 
Newly, Ced\'o, Jespers, and Verwimp \cite{CeJeVe20x} investigated the structure monoid of arbitrary solutions not necessarily bijective, inspired by the work of Gateva-Ivanova and Majid \cite{GaMa08}.
Finally, Castelli, Catino, and Stefanelli \cite{CaCaSt20x}
developed a theory of extensions for left non-degenerate solutions involving the algebraic structure of $q$-cycle set.

In this paper, we introduce a new algebraic structure, namely the left inverse semi-brace, which turns out to be a useful tool for determining solutions not necessarily bijective.
In particular, this notion involves inverse semigroups, which we recall to be semigroups $S$ such that, for each $x \in S$, there exists a unique $x^{-1} \in S$ satisfying $xx^{-1}x = x$ and $x^{-1}xx^{-1} = x^{-1}$. Inverse semigroup theory was initiated in the 1950s and it has been extensively studied during the years. Most of the known results up to the early 1980s are summarized into the monograph by Petrich \cite{Pe84}. Many works have been appeared on inverse semigroups until now and  whole chapters of classical semigroups books have been dedicated to this topic, like that by Howie \cite{Ho95}. One can draw on to the large and currently investigated theory of inverse semigroups to give a new research perspective to the open problem of finding solutions.\\
A triple $(S, +, \cdot)$ is said to be a \emph{left inverse semi-brace} if $(S,+)$ is a semigroup, $(S, \cdot)$ an inverse semigroup, and it holds 
\begin{align*}
a \left(b + c\right) = a  b + a\left(a^{-1} + c\right), 
\end{align*}
for all $a, b,c \in S$. This structure includes left semi-braces, introduced by Catino, Colazzo, and Stefanelli in \cite{CaCoSt17} and Jespers and Van Antwerpen in \cite{JeAr19}, where the semigroup $\left(S, \cdot\right)$ is a group, and generalized semi-braces treated by the first authors in  \cite{CCoSt20x-2} with $(S, \cdot)$ a Clifford semigroup. 
If $(S,+,\cdot)$ is a left semi-brace with $\left(S,+\right)$ a left cancellative semigroup, then the map associated to $S$, i.e., the map $r_S:S \times S \to S \times S$ given by
\begin{align*}
    r_S\left(a,b\right)
    = \left(a\left(a^{-1} + b\right), \, \left(a^{-1} + b\right)^{-1}b\right),
\end{align*}
for all $a,b\in S$, is always a solution that is also left non-degenerate \cite{CaCoSt17}.
Instead, if $S$ is an arbitrary left semi-brace, the map $r_S$ is not a solution in general and, in this context, a characterization has been given in \cite[Theorem 3]{CCoSt20x-2}. In light of this, we provide sufficient conditions in order that the map $r_S$ associated to an inverse semi-brace $S$ is a solution.\\
Our attention is turned to the study of inverse semi-braces $S$ for which the map $r_S$ is a solution. Specifically, we show how these structures produce a variety of new solutions which are non-bijective and degenerate.
Just to give an idea, fixed a right zero semigroup $\left(S, +\right)$,  if $(S, \cdot)$
is an arbitrary inverse semigroup, then
the map $r_S$ associated to the left inverse semi-brace $(S, +, \cdot)$, that is given by $r_S\left(a,b\right)=\left(ab,b^{-1}b\right)$, is an idempotent and two-sided degenerate solution.
Moreover, there are as many solutions $r_S$ as there are inverse semigroups $(S, \cdot)$. Thus, the number of finite inverse semigroups allows for determining a lower bound for finite idempotent solutions of the form of $r_S$. In this respect, we refer the reader to \cite{MaMa19}, where it has been provided an algorithm for the enumeration of inverse semigroups of finite order. 
We highlight that if the semigroup $(S, \cdot)$ is completely regular, hence $(S,+,\cdot)$ is a generalized semi-brace, such a map $r_S$ is not a solution, in general.

The article is structured as it follows. The first two sections are devoted to introducing left inverse semi-braces and some easy examples. Although the description of the semigroup $(S,+)$ of a left inverse semi-brace $(S,+, \cdot)$ is rather complicated, we show that in the specific case of $S$ a left semi-brace, $(S,+)$ is a rectangular semigroup. Furthermore, we focus on left inverse semi-braces which give solutions. Specifically, we find sufficient conditions to obtain them and show that, if $S$ is a left semi-brace, then the condition provided in the characterization \cite[Theorem 3]{CCoSt20x-2} is satisfied.\\
In the remainder of the work, given two left inverse semi-braces $S$ and $T$, we present some constructions of left inverse semi-braces having the Cartesian Product $S\times T$ as underlying set. Into detail, in the third section we extend the matched product of left semi-braces, contained in \cite{CCoSt20}, to the matched product of left inverse semi-braces. Moreover, we demonstrate that also in this case any matched product of $S$ and $T$ for which the maps $r_S$ and $r_T$ are solutions gives rise to a new solution that is exactly the matched product of $r_S$ and $r_T$. As an application, we show that already the easier case of the semidirect product leads to find various and new examples of solutions, of which we determine their order.\\
A new construction of left inverse semi-braces, including the semidirect product of left inverse semi-braces is introduced in the fourth section and we call it the double semidirect product. Assuming that the maps $r_S$ and $r_T$ associated to $S$ and $T$, respectively, are solutions, we aim to provide a lot of examples of solutions of various kinds associated to the double semi-direct product of $S$ and $T$, not only in the specific class of the degenerate and non-bijective ones. 
In particular, if $S$ and $T$ are arbitrary left semi-braces or skew left braces, we show that easy and usable conditions allow for obtaining a new solution on $S\times T$.\\
In the last section, we extend the asymmetric product of left cancellative left semi-braces given in \cite{CaCoSt17} to left inverse semi-braces. 
Let us note that the asymmetric product of left braces proved useful to obtain rather systematic constructions of regular subgroups of the affine group \cite{CCoSt16}, to investigate simple left braces \cite{BaCeJeOk19,CeJeOk20-2, CeJeOk21}, and it has been related to another construction of finite braces, the
upper shifted semi-direct product of braces, in a recent work \cite{Ru20-3}. To our purpose, we need the notion of $\delta$-cocycle on semigroups, inspired to that used for groups (see the Schreier’s extension in \cite[Theorem 15.1.1]{Ha59}). 
Let us observe that this is not a simple readjustment of the definition introduced in \cite{CaCoSt17}, since it involves entirely the additive structures of the left inverse semi-braces.
We privilege again left inverse semi-braces having solutions and we provide sufficient conditions to obtain solutions in the case of arbitrary left semi-braces.

\bigskip

\section{Left inverse semi-braces: definitions and examples}
In this section, we introduce the definition of left inverse semi-brace and we give some basic examples. 
Further and more challenging examples will be presented later.

\medskip

For the ease of the reader, we initially recall essential notions on inverse semigroups for our treatment. For further details one can see the book by Howie \cite{Ho95} or the monograph by Petrich \cite{Pe84}. A semigroup $S$ is called an \emph{inverse semigroup} if, for each $a\in S$, there exists a unique element $a^{-1}$ of $S$ such that $aa^{-1}a=a$ and $a^{-1}aa^{-1}=a^{-1}$. We call such an element $a^{-1}$ the \emph{inverse} of $a$.\\
Evidently, every group is an inverse semigroup. The behaviour of inverse elements in an inverse semigroup is similar to that in a group, as we recall below.
\begin{lemma}
Given an inverse semigroup $S$, they hold
 $\; (ab)^{-1}=b^{-1}a^{-1}$ \, and \, $(a^{-1})^{-1}=a$,\, for all $a,b \in S$.
\end{lemma}
\begin{lemma}
If $S$ and $T$ are inverse semigroups and $\Phi$ is a homomorphism from $S$ into $T$, then $\Phi(a^{-1}) = \Phi(a)^{-1}$, \, for any $a\in S$.
\end{lemma}

\medskip

Note that if $a$ is an element of an inverse semigroup $S$, then $aa^{-1}$ and $a^{-1}a$ are idempontent elements of $S$. Moreover, the set $\E(S)$ of the idempotents  is a commutative subsemigroup of $S$ and $e=e^{-1}$, for every $e \in \E(S)$.\\
Finally, we recall that an inverse semigroup in which its idempotent elements are central is called a \emph{Clifford semigroup}.

\bigskip

Now, we give the notion of left inverse semi-brace.
	\begin{defin}
        Let $S$ be a set with two operations $+$ and $\cdot$ such that $\left(S,+\right)$ is a semigroup (not necessarily commutative) and $\left(S,\cdot\right)$ is an inverse semigroup. Then, we say that $\left(S, + , \cdot \right)$ is \emph{a left inverse semi-brace} if
        \begin{align}\label{key1}
            a \left(b+c\right) = ab + a\left(a^{-1} +c\right)
        \end{align}
        holds, for all $a, b, c \in S$. 
        We call $(S,+)$ and $(S,\cdot)$ the \emph{additive semigroup} and the 	\emph{multiplicative semigroup} of $S$, respectively.\\
        A right inverse semi-brace is defined similarly, by replacing condition \eqref{key1} with $\;\left(a+b\right)c =  \left(a+c^{-1}\right)c + bc,\;$
        for all $a,b,c\in S$. \\
        A two-sided inverse semi-brace $(S, +, \cdot)$ is a left inverse semi-brace that is also a right inverse semi-brace with respect to the same operations $+$ and $\cdot$.
    \end{defin} 
\medskip

Clearly,  any left semi-brace is a left inverse semi-brace, since in this case $\left(S,\cdot\right)$ is a group.
Other examples of left inverse semi-braces are generalized left semi-braces $(S,+,\cdot)$, introduced in \cite{CCoSt20x-2}, with $\left(S,\cdot\right)$ a Clifford semigroup. For the ease of the reader, we recall that $(S, +, \cdot)$ is a generalized semi-brace if $(S,+)$ is a semigroup, $\left(S, \cdot\right)$ is a completely regular semigroup and condition \eqref{key1} is satisfied.

Any arbitrary inverse semigroup gives easily rise to left inverse semi-braces, as we will show in the next examples. 

\begin{ex}\label{trivial-semibrace}
If $(S,\cdot)$ is an inverse semigroup and $(S,+)$ is a right zero semigroup or a left zero semigroup, then $S$ is an inverse two-sided semi-brace, which we call \emph{trivial inverse semi-brace}. Clearly, if $|S|>1$, then such trivial semi-braces are not isomorphic.
\end{ex}
\medskip

\begin{ex}\label{es-nuovo}
Let  $(S,\cdot)$ be an inverse semigroup and set $a + b = aa^{-1}b$, for all $a,b\in S$. Then, $S$ is a left inverse semi-brace. Note that if $(S, \cdot)$ is a Clifford semigroup, then $S$ is a two-sided inverse semi-brace.
Similarly, the same is true if we consider the opposite sum, i.e., $a+b=bb^{-1}a$, for all $a,b \in S$.
\end{ex}

\medskip
\begin{ex}\label{ex:prod-idempot} 
Let $(S, \cdot)$ be an inverse semigroup, $e \in \E(S)$ and set $a+b=b e$, for all $a,b\in S$. Then, it is easy to check that $S$ is a left inverse semi-brace.  
Note that, if $e$ is central, then $(S,+.\cdot)$ is also a right inverse semi-brace. 
\end{ex}
\medskip

The following are examples of left inverse semi-braces obtained starting from an arbitrary Clifford semigroup.
\begin{exs}\label{ex:first-examples} 
Let $(S, \cdot)$ be a Clifford semigroup.
If $a + b = a  b$, for all $a,b \in S$, then $S$ is a generalized two-sided semi-brace (see \cite[Example 9]{CCoSt20x-2}). The same is true if we take the opposite sum, i.e., $a + b = b a$, for all $a,b \in S$.
\end{exs}

\medskip

Now, we present a construction of left inverse semi-braces which involves strong semilattice of inverse semigroups, that it is well-known to be inverse (see \cite[Ex.(ii), p.90]{Pe84}).
Note that a similar construction has already been considered in the case of generalized left semi-braces in \cite[Proposition 10]{CCoSt20x-2}. For this reason, we omit the detailed proof.

\begin{prop}\label{prop:Strong-Lattice-Inverse-Semi-Brace}
Let $Y$ be a (lower) semilattice, $\left\{S_{\alpha}\ \left|\ \alpha \in S\right.\right\}$ a family of disjoint left inverse semi-braces. For each pair $\alpha,\beta$ of elements of $Y$ such that $\alpha \geq \beta$, let $\phii{\alpha}{\beta}:\s{\alpha}\to \s{\beta}$ be a homomorphism of left inverse semi-braces such that
\begin{enumerate}
    \item $\phii{\alpha}{\alpha}$ is the identical automorphism of $\s{\alpha}$, for every $\alpha \in Y$,
    \item $\phii{\beta}{\gamma}{}\phii{\alpha}{\beta}{} = \phii{\alpha}{\gamma}{}$, for all $\alpha, \beta, \gamma \in S$ such that $\alpha \geq \beta \geq \gamma$.
\end{enumerate}
Then, $S = \bigcup\left\{\s{\alpha}\ \left|\ \alpha\in Y\right.\right\}$ endowed by the addition and the multiplication defined by
\begin{align*}
    a+b&:= \phii{\alpha}{\alpha\beta}(a)+\phii{\beta}{\alpha\beta}(b),\\
     a\,b&:= \phii{\alpha}{\alpha\beta}(a)\,\phii{\beta}{\alpha\beta}(b),
\end{align*}
for every $a\in \s{\alpha}$ and $b\in \s{\beta}$, is a left inverse semi-brace. Such a left inverse semi-brace is said to be the \emph{strong semilattice $S$ of left inverse semi-brace $S_{\alpha}$} and is denoted by $S=[Y; S_\alpha,\phii{\alpha}{\beta}]$.
\end{prop}

\medskip

The previous examples and Proposition \ref{prop:Strong-Lattice-Inverse-Semi-Brace} suggest that the structure of additive semigroup of a left inverse semi-brace is rather complicated. At present, only partial results are known for left semi-braces. 
Under mild assumptions, for instance in the finite case, it has been proved that the additive semigroup $\left(S, +\right)$ of a left semi-brace $S$ is a completely simple semigroup (see \cite[Theorem 2.8]{JeAr19}). 
In more detail, $\left(S, +\right)$ is a rectangular group, i.e., it is isomorphic to the direct product of a group and a rectangular band, see \cite[Theorem 3]{CCoSt20x-2}. 

\medskip

Now, we show that the additive semigroup of an arbitrary left semi-brace is a rectangular semigroup. In this regard, we recall  that a semigroup $\left(S, +\right)$ is \emph{stationary on the right} if
\begin{align*}
    a + b = a + c \, \Longrightarrow \, x + b = x + c
\end{align*}
holds, for all $a,b,c,x\in S$. By \cite[Ex. 7, p. 98]{ClPr61}, any semigroup which is stationary on the right is rectangular. We recall that a semigroup $\left(S, +\right)$  is \emph{rectangular} if it holds
$$
a+x=b+x=a+y\, \Longrightarrow \, a+x=b+y,
$$
for all $a,b,x,y\in S$ (see \cite[Definition III.5.14]{PeRe99}).
In addition, 
we recall that if $S$ is a left semi-brace, the identity $1$ of the group $(S,\cdot)$ satisfies two special properties, namely, 
$1$ is an idempotent element of $(S,+)$ and it is also a middle unit, i.e., $a + 1 + b = a + b$,
for all $a, b\in S$ (see Lemma 2.4 (1) and Lemma 2.6 (ii) of \cite{JeAr19}).

\begin{prop}\label{prop-middleunit-clifford-preston}
Let $(S,+,\cdot)$ be a left semi-brace. Then, the semigroup $\left(S,+\right)$ is stationary on the right. Consequently, $\left(S,+\right)$ is a rectangular semigroup and they hold:
\begin{enumerate}
    \item  $E\left(S\right)$ is a rectangular band;
    \item  $e\in E\left(S\right)$ if and only if $e$ is a middle unit.
\end{enumerate}

\begin{proof}
Let $a,b,c\in S$ such that $a+b=a+c$. Then, we obtain
\begin{eqnarray}
a+b=a+c &\Longrightarrow & a(a^{-1}+a^{-1}(a+b))=a(a^{-1}+a^{-1}(a+c))\nonumber\\
&\Longrightarrow& 1+b=1+c.\nonumber
\end{eqnarray}
Hence, since $1$ is a middle unit, we have that $x+b=x+1+b=x+1+c=x+c$,
for every $x\in S$.
Thus, the semigroup $(S, +)$ is stationary on the right and so it is rectangular.\\
Now, by \cite[Ex. 7, p. 98]{ClPr61} the set of idempotents $\E(S)$ is a rectangular band and any idempotent is a middle unit. Vice versa,  \cite[Proposition 4]{CCoSt20x-2} completes the proof.
\end{proof}
\end{prop}

\medskip

Finally, we note that, under specific assumptions on the additive structure, a left inverse semi-brace $(S,+,\cdot)$ is necessarily a skew left brace, as shown below.
In this respect, we recall that if $(S,+)$ is left cancellative, then $1$ is also a left identity (see \cite[p. 165]{CaCoSt17}).
\begin{rem}
Let $(S,+,\cdot)$ be a left inverse semi-brace. If $(S, +)$ is a left cancellative semigroup with a right identity $1$, then $(S, \cdot)$ is a group. Indeed, if $a \in S$, then
 \begin{align*}
     a1 + 1= 1 = a(1+1)=a1+a(a^{-1}+1)=a1+aa^{-1},
 \end{align*}
 hence $aa^{-1}=1$. 
 Therefore, since every idempotent of $S$ can be expressed in the form $xx^{-1}$, with $x \in S$, we obtain that $1$ is the unique idempotent in $(S, \cdot)$. 
 Therefore, $S$ is a left cancellative left semi-brace. Since, in this case $1$ is a left identity in  $(S, +)$, it follows that $(S, +)$ is a monoid. As observed in \cite[p. 167]{CaCoSt17}, the structure $S$ is necessarily a skew left brace.
\end{rem}

\medskip

\bigskip

\section{Solutions associated to left inverse semi-braces}

In this section, we deal with solutions associated to left inverse semi-braces and we provide sufficient conditions to obtain them. In this way, we give   several solutions which are associated to the examples in the previous section.

\medskip

As is common in the semi-brace theory, given a left inverse semi-brace $(S, +, \cdot)$, let us consider the two maps $\lambda: S\to \End(S,+), \,a\mapsto\lambda_a$ from $S$ into the endomorphism semigroup  of $(S, +)$ and $\rho: S\to S^S, \,b\mapsto\rho_b$ from $S$ into the semigroup of the maps from $S$ into itself defined by
\begin{align*}
    \lambda_a(b) = a(a^{-1} + b) \qquad \qquad \rho_b(a) = (a^{-1} + b)^{-1}b,
\end{align*}
for all $a,b\in S$, respectively. Moreover, if $a,b,c\in S$, then  $\lambda_{ab}\left(c\right) 
= abb^{-1}a^{-1} + \lambda_a\lambda_b\left(c\right)$, where
we observe that $abb^{-1}a^{-1}\in \E(S)$. We call the map $r_S:S\times S\to S\times S$ given by
\begin{align*}
    r_S(a,b) = (\lambda_a(b), \rho_b(a)),
\end{align*}
for all $(a,b)\in S\times S$, the \emph{map associated to the left inverse semi-brace $(S,+,\cdot)$}.
Note that, if $(S,+,\cdot)$ is a left semi-brace with $(S,+)$ a left cancellative semigroup, then the map $r_S$ is a left non-degenerate solution (see \cite[Theorem 9]{CaCoSt17}). Let us recall that not every left semi-brace gives rise to solutions, yet. In this context, in \cite{CaCoSt19} a characterization has been provided.
\begin{theor}[Theorem 3, \cite{CaCoSt19}]\label{th-generalized-sol}
        Let $\left(S,+,\cdot\right)$ be a left semi-brace. The map $r_S$ associated to $S$ is a solution if and only if
        \begin{align}\label{eq:condsolution}
            a + \lambdaa{b}{\left(c\right)}\left(1 + \rhoo{c}{\left(b\right)}\right) = a + b \left(1+c\right)
        \end{align}
        holds, for all $a,b,c\in S$.
\end{theor}
\noindent Observe that, as shown in \cite[p. 8]{CCoSt20x-2},  if the map $\rho$ is an anti-homomorphism from the group $(S,\cdot)$ into the monoid of the maps from $S$ into itself, then the map $r_S$ associated to the left semi-brace $(S,+,\cdot)$ satisfies condition \eqref{eq:condsolution}. In particular, if $(S,+)$ is a left cancellative semigroup, the condition \eqref{eq:condsolution} is satisfied, too.

\medskip

In the following, we provide sufficient conditions to obtain solutions through left inverse semi-braces.

\begin{theor}\label{th:sol-inverse-semi-brace}
Let $(S,+,\cdot)$ be a left inverse semi-brace and $r_S$ the map associated to $S$. If the following are satisfied
\begin{enumerate}
    \item $\left(a + b\right)\left(a + b\right)^{-1}\left(a + bc\right) = a + bc$ 
    \vspace{1mm}
     \item $\lambda_a(b)^{-1}+\lambda_{\rho_b(a)}\left(c\right)
     = \lambda_a(b)^{-1}+\lambda_{\left(a^{-1}+\, b\right)^{-1}}\lambda_b\left(c\right)$
   \vspace{1mm}
    \item $\rho_b\left(a\right)^{-1}+c=\left(b^{-1}+c\right)\left(\rho_{\lambda_b\left(c\right)}\left(a\right)^{-1}+\rho_c\left(b\right)\right)$, 
\end{enumerate}
for all $a,b,c \in S$, then the map $r_S$ is a solution.
\begin{proof}
It is a routine computation to verify that the map $r$ associated to $S$ given by $r(a,b)=(\lambda_a(b), \rho_b(a))$ is a solution if and only if they hold
\begin{align*}
    &\lambda_a\lambda_b(c)=\lambda_{\lambda_a\left(b\right)}\lambda_{\rho_b\left(a\right)}\left(c\right)\\
    &\lambda_{\rho_{\lambda_b\left(c\right)}\left(a\right)}\rho_c\left(b\right)=\rho_{\lambda_{\rho_b\left(a\right)}\left(c\right)}\lambda_a\left(b\right)\\
    &\rho_c\rho_b(a)=\rho_{\rho_c\left(b\right)}\rho_{\lambda_b\left(c\right)}\left(a\right),
\end{align*}
for all $a,b,c \in S$. Thus, if $a,b,c \in S$, we have that
\begin{align*}
    \lambda_{\lambda_a\left(b\right)}\lambda_{\rho_b\left(a\right)}\left(c\right)&=\lambda_a\left(b\right)\left(\lambda_a\left(b\right)^{-1}+\lambda_{\rho_b\left(a\right)}\left(c\right)\right)\\
    &=\lambda_a\left(b\right)\left(\lambda_a\left(b\right)^{-1}+\lambda_{\left(a^{-1}+\, b\right)^{-1}}\lambda_b\left(c\right)\right) &\mbox{by 2.}\\
    &=\lambda_a\left(b\right)\left(\left(a^{-1}+b\right)^{-1}a^{-1}+\left(a^{-1}+b\right)^{-1}\left(a^{-1}+b+\lambda_b\left(c\right)\right)\right)\\
    &=\lambda_a\left(b\right)\left(a^{-1}+b\right)^{-1}\left(a^{-1}+\lambda_b\left(c\right)\right) &\mbox{by \eqref{key1}}\\
    &=a\left(a^{-1}+b\right)\left(a^{-1}+b\right)^{-1}\left(a^{-1}+b\left(b^{-1}+c\right)\right)\\
    &= a\left(a^{-1}+b\left(b^{-1}+c\right)\right)&\mbox{by 1.}\\
    &=a\left(a^{-1}+\lambda_b\left(c\right)\right)=\lambda_a\lambda_b\left(c\right).
\end{align*}
Moreover, we obtain
\begin{align*}
    \lambda_{\rho_{\lambda_b\left(c\right)}\left(a\right)}\rho_c\left(b\right)&=\rho_{\lambda_b\left(c\right)}\left(a\right)\left(\rho_{\lambda_b\left(c\right)}\left(a\right)^{-1}+\rho_c\left(b\right)\right)\\
    &=\left(a^{-1}+\lambda_b\left(c\right)\right)^{-1}\lambda_b\left(c\right)
    \left(\rho_{\lambda_b\left(c\right)}\left(a\right)^{-1}+\rho_c\left(b\right)\right) \\
    &=\left(a^{-1}+\lambda_b\left(c\right)\right)^{-1}b\left(b^{-1}+c\right)
    \left(\rho_{\lambda_b\left(c\right)}\left(a\right)^{-1}+\rho_c\left(b\right)\right)\\
    &=\left(a^{-1}+\lambda_b\left(c\right)\right)^{-1}b\left(\rho_b\left(a\right)^{-1}+c\right)&\mbox{by 3.}\\
    &=\left(a^{-1}+\lambda_b\left(c\right)\right)^{-1}\left(a^{-1}+b\right)\left(a^{-1}+b\right)^{-1}b\left(\rho_b\left(a\right)^{-1}+c\right) &\mbox{by 1.}\\
    &=\left(\left(a^{-1}+b\right)^{-1}\left(a^{-1}+\lambda_b\left(c\right)\right)\right)^{-1}\rho_b\left(a\right)\left(\rho_b\left(a\right)^{-1}+c\right)\\
    &=\left(\left(a^{-1}+b\right)^{-1}a^{-1}+\lambda_{\left(a^{-1}+b\right)^{-1}}\lambda_b\left(c\right)\right)^{-1}\lambda_{\rho_b\left(a\right)}\left(c\right) &\mbox{by \eqref{key1}}\\
    &=\left(\lambda_a\left(b\right)^{-1}+\lambda_{\left(a^{-1}+b\right)^{-1}}\lambda_b\left(c\right)\right)^{-1}\lambda_{\rho_b\left(a\right)}\left(c\right)\\
    &= \left(\lambda_a\left(b\right)^{-1}+\lambda_{\rho_b\left(a\right)}\left(c\right)\right)^{-1}\lambda_{\rho_b\left(a\right)}\left(c\right)&\mbox{by 2.}\\
    &= \rho_{\lambda_{\rho_b\left(a\right)}\left(c\right)}\lambda_a\left(b\right).
\end{align*}
Finally, we get
\begin{align*}
    \rho_c\rho_b\left(a\right)&=\left(\rho_b\left(a\right)^{-1}+c\right)^{-1}c\\
    &=\left(\left(b^{-1}+c\right)\left(\rho_{\lambda_b\left(c\right)}\left(a\right)^{-1}+\rho_c\left(b\right)\right)\right)^{-1}c &\mbox{by 3.}\\
    &=\left(\rho_{\lambda_b\left(c\right)}\left(a\right)^{-1}+\rho_c\left(b\right)\right)^{-1}\left(b^{-1}+c\right)^{-1}c\\
    &=\left(\rho_{\lambda_b\left(c\right)}\left(a\right)^{-1}+\rho_c\left(b\right)\right)^{-1}\rho_c\left(b\right)\\
    &=\rho_{\rho_c\left(b\right)}\rho_{\lambda_b\left(c\right)}\left(a\right).
\end{align*}
Therefore, the map $r_S$ is a solution on the left inverse semi-brace $S$.
\end{proof}
\end{theor}

\smallskip

\begin{rem}
If $(S,+,\cdot)$ is a left semi-brace, then the condition $1.$ in \cref{th:sol-inverse-semi-brace} trivially holds. Moreover, since $\lambda$ is a homomorphism from $\left(S,\cdot\right)$ into $\End\left(S,+\right)$, we have that
condition $2.$ is clearly satisfied.  In addition, since in this case $\left(S,\cdot\right)$ is a group, we obtain that 
 \begin{align*}
    \rho_b\left(a\right)^{-1}+c
    &=b^{-1}b\left(b^{-1}\left(a^{-1} + b\right) + c\right)\\
    &=  b^{-1}\left(a^{-1} + b + \lambda_b\left(c\right)\right)
    = b^{-1}\left(a + b\left(1 + c\right)\right)
\end{align*}
and
\begin{align*}
    &\left(b^{-1}+c\right)\left(\rho_{\lambda_b\left(c\right)}\left(a\right)^{-1}+\rho_c\left(b\right)\right)\\     
    &= b^{-1}b\left(b^{-1} + c\right)
    \left(\lambda_{b}\left(c\right)^{-1}\left(a^{-1} + \lambda_b\left(c\right)\right) + \rho_c\left(b\right)\right)\\
    &= b^{-1}\lambda_b\left(c\right)
    \left(\lambda_{b}\left(c\right)^{-1}\left(a^{-1} + \lambda_b\left(c\right)\right) + \rho_c\left(b\right)\right)\\
    &= b^{-1}\left(a^{-1} + \lambda_b\left(c\right) + \lambda_{\lambda_b\left(c\right)}\rho_b\left(c\right)\right)\\
    &= b^{-1}\left(a + \lambda_{b}\left(c\right)\left(1 + \rho_c\left(b\right)\right)\right),
 \end{align*} 
hence the condition $3.$ in \cref{th:sol-inverse-semi-brace} holds if and only if \eqref{eq:condsolution} in \cref{th-generalized-sol} is satisfied.
\end{rem}
\medskip

It is a routine computation to verify that all the examples of left inverse semi-braces provided until now satisfy the conditions of \cref{th:sol-inverse-semi-brace}, hence the maps associated to each left inverse semi-brace are solutions. 
Below, we list these solutions and we highlight some properties about their behavior. In general, they are not bijective and lie in the class of degenerate ones.
\medskip

\begin{exs}\label{ex:sol}
Let $(S, \cdot)$ be an arbitrary inverse semigroup. If $S$ is the trivial left inverse semi-brace in \cref{trivial-semibrace} with $(S, +)$ a right zero semigroup, the map $r_S$ associated to $S$ given by
    \begin{align*}
        r_S(a,b)=(ab, b^{-1}b),
    \end{align*}
    for all $a,b \in S$, is an idempotent solution. Similarly, if $S$ is the trivial left inverse semi-brace with $(S, +)$ a left zero semigroup, we get the idempotent solution
    \begin{align*}
        r_S(a,b)=(aa^{-1}, ab),
    \end{align*}
    for all $a,b \in S$.
\end{exs}
\noindent Note that if $|S| > 1$ such solutions are not isomorphic in the sense of \cite{CeJeOk14}. 
In addition, since they are strictly linked to a given inverse semigroup, it is clear that the number of such semigroups determines a lower bound for idempotent solutions. In this regard, we highlight that, recently, in \cite{MaMa19} it has been provided an algorithm for the enumeration of the inverse semigroups of order $n$ (up to isomorphism).  Furthermore, these maps are added to the class of idempotent solutions which are known until now \cite{Le17, MaSh18, Charl19, CCoSt20, CCoSt20x-2}.

\medskip

\begin{rem}
Note that if $(S, \cdot)$ is a Clifford semigroup, the map $r(a,b)=(ab, b^{-1}b)$ associated to the trivial left inverse semi-brace in \cref{ex:sol}, where the first component is exactly the multiplication in the semigroup $S$, is still a solution to the quantum Yang-Baxter equation, i.e., it satisfies the relation $r_{12}r_{13}r_{23}=r_{23}r_{13}r_{12}$, and to the pentagon equation, i.e., it holds $r_{23}r_{13}r_{12}=r_{12}r_{23}$. Thus, it belongs to the class of solutions of pentagonal type (for more details see \cite[Proposition 8]{CaMaSt20}). Similarly, recalling that a map $r$ is a solution if and only if $\tau r$ is a solution of the quantum Yang-Baxter equation, with $\tau$ the flip map on $S \times S$, then the map $\tau r(a,b)=(ab, aa^{-1})$ is a solution of pentagonal type, too.
\end{rem} 

\medskip

\begin{exs}
The map $r_S$ associated to the left inverse semi-braces $S$ in \cref{es-nuovo} with $a +b = aa^{-1}b$, for all $a,b \in S$, given by
    \begin{align*}
        r_S(a,b)=\left(ab,ab\left(ab\right)^{-1}\right),
    \end{align*}
    for all $a,b\in S$, is an idempotent solution. Analogously, the map $t_S$ associated to the left inverse semi-braces $S$ in \cref{es-nuovo}  with  $a +b = aa^{-1}b$, for all $a,b \in S$, defined by
    \begin{align*}
        t_S(a,b)=\left(ab\left(ab\right)^{-1},ab\right),
    \end{align*}
    for all $a,b\in S$, is an idempotent solution.
Let us note that $t_S = \tau r_S$, consequently these two solutions are not isomorphic.
\end{exs}

\medskip
\begin{ex}\label{ex-soluzione-semibrace-be} Let $S$ be the left inverse semi-brace in \cref{ex:prod-idempot} where $a + b = b\cdot e$, with $e \in \E(S) \cap \zeta(S)$, then the map 
    \begin{align*}
        r_S(a,b)=(abe, eb^{-1}b),
    \end{align*}
    for all $a,b\in S$, is an idempotent solution on $S$. Moreover, let us observe that the solution $t_S$ associated to $S$ considered as right inverse semi-brace is given by
\begin{align*}
    t_S\left(a,b\right)=\left(a\left( a+b^{-1}\right)^{-1}, \left(a+b^{-1}\right)b\right)=\left(abe, eb^{-1}b\right)
    = r_S\left(a,b\right),
\end{align*}
for all $a,b\in S$.
\end{ex}

\medskip

\begin{exs}\label{ex-sol-semibrace-clifford}
The map $r_S$ associated to the left inverse semi-brace $S$ in \cref{ex:first-examples} with $a+b=ab$, for all $a,b \in S$, given by
    \begin{align*}
        r_S(a,b) = (aa^{-1}b, b^{-1}ab),
    \end{align*}
    for all $a,b\in S$, is a solution. If in addition the Clifford semigroup $(S,\cdot)$ is commutative, then $r_S$ is a cubic solution, i.e., $r_S^3 = r_S$, see \cite[p. 12]{CCoSt20x-2}.
    Similarly, the solution associated to the left inverse semi-brace $S$ in \cref{ex:first-examples} with $a+b=ba$, for all $a,b \in S$, is given by
    \begin{align*}
        r_S(a,b) = (aba^{-1}, ab^{-1}b),
    \end{align*}
    for all $a,b\in S$, and also in this case $r_S$ is cubic if $(S, \cdot)$ is commutative.
\end{exs}

\bigskip

As mentioned above, if $S=[Y; S_\alpha,\phii{\alpha}{\beta}]$ is a strong semilattice of left inverse semi-braces for which every $S_\alpha$ has  $r_{S_{\alpha}}$ as a solution, then the map 
associated to $S$
is a solution. 
This is a consequence of the technique named strong semilattice of solutions introduced in \cite[Theorem 12]{CCoSt20x-2}.

\medskip
\begin{theor}
	Let $S=[Y; S_\alpha,\phii{\alpha}{\beta}{}]$ be  a strong semilattice of left inverse semi-braces. If $\s{\alpha}$ has $r_{\alpha}$ as a solution, for every $\alpha \in Y$, then the map $r_S$ given by
	\begin{align*}
	r_S\left(x, y\right):= 
	r_{\alpha\beta}\left(\phii{\alpha}{\alpha\beta}\left(x\right),
	\phii{\beta}{\alpha\beta}\left(y\right) \right),
	\end{align*}	
	for all $x\in S_{\alpha}$ and $y\in S_{\beta}$, is a solution on $S$. 
\end{theor}

\medskip

In the remainder of this work, we provide constructions of left inverse semi-braces which allow to determine solutions starting from known left inverse semi-braces $S$ having $r_S$ as solution. In the specific, we will focus on constructions on the Cartesian product of left inverse semi-braces.
\bigskip

\section{The matched product of left inverse semi-braces}

This section is devoted to extend the construction of the matched product of left semi-braces contained in \cite{CCoSt20} to the class of the left inverse semi-braces. 
Moreover, we show that any matched product of two left inverse semi-braces $S$ and $T$ for which the maps $r_S$ and $r_T$ are solutions gives rise to a new solution that is exactly the matched product of $r_S$ and $r_T$. Finally, we provide various examples of solutions in the easier case, the semidirect product, and we highlight some properties about their order.

\medskip

We begin by reminding the classical \emph{Zappa product} of two semigroups contained in \cite{Ku83}. Let $S$ and $T$ be semigroups, $\sigma:T\to S^{S}$ and $\delta:S\to T^{T}$ maps, and set $^u a:= \sigma\left(u\right)\left(a\right)$ and $u^a:= \delta\left(a\right)\left(u\right)$, for all $a\in S$ and $u\in T$. If the following conditions are satisfied
	\begin{align}\label{S1}
	&^{u}(ab)=\,^ua \, ^{u^a}b	
	&^{uv}a = \, ^u(^va) \tag{Z1}\\
	\label{S2}&( uv)^a = u^{^va} \,\, v^a&u^{ab}=(u^a)^b\tag{Z2}
				\end{align}
for all $a,b \in S$ and $u,v \in T$, then $S \times T$ is a semigroup with respect to the operation defined by
	\begin{equation}\label{prod-semigruppo-match}
	\left(a,u\right)\left(b,v\right)=\left(a\,^ub, u^b\,v\right),
	\end{equation}
	for all $a,b \in S$ and $u,v \in T$. In the following theorem, we give the
necessary conditions for Zappa product of inverse semigroups to be inverse. Hereinafter, for the ease of the reader, we use the letters
$a,b,c$ for $S$ and $u,v,w$ for $T$.

\begin{theor}\emph{(cf. \cite[Theorem 4]{Wa15})}\label{inverse-semigroup-wazzan}
Let $S$ and $T$ be inverse semigroups, $\sigma:T\to S^{S}$ and $\delta:S\to T^{T}$ maps satisfying \eqref{S1} and \eqref{S2}. If $\sigma(T)\subseteq \Aut(S)$ and $\delta(S)\subseteq \Aut(T)$ and
$$
a \, {}^ua = a, \, u^a \, u = u\ \Longrightarrow \ {}^ua=a, \, u^a=u
$$
holds, for all $a\in S$ and $u\in T$, then $S\times T$ endowed with the operation \eqref{prod-semigruppo-match} is an inverse semigroup. 
\end{theor}

\medskip

Coherently with the notation of the matched product used in the context of left semi-braces in \cite{CCoSt20x-2}, we introduce suitable maps $\alpha$ and $\beta$ which allow for obtaining a new left inverse semi-brace having the multiplicative semigroup isomorphic to a Zappa product of the starting inverse semigroups.
\begin{defin}\label{def:mps-inv}
    Let $S$ and $T$ be left inverse semi-braces,   $\alpha: T \to \Aut\left(S\right)$ a homomorphism of inverse semigroups from $\left(T,\cdot\right)$ into the automorphism group of $\left(S,+\right)$, and $\beta:S\to \Aut\left(T\right)$ a homomorphism of inverse semigroups from $\left(S,\cdot\right)$ into the automorphism group of $\left(T,+\right)$ such that
    	\begin{align}
    		\label{eq:mps1}
    		\alphaa{}{u}{\left(\alphaa{-1}{u}{\left(a\right)}\, b\right)} = a\, \alphaa{}{\beta^{-1}_{a}{\left(u\right)}}{\left(b\right)}\qquad\quad 
    		\beta_{a}\left(\beta^{-1}_{a}\left(u\right)\, v\right) = u\, \beta_{\alphaa{-1}{u}{\left(a\right)}}\left(v\right)
    		\end{align}
    		\begin{align}\label{eq:mps-idemp}
    		\alphaa{}{u}{\left(\alphaa{-1}{u}{\left(a\right)}\, a\right)} = a, \, \, \
    		\beta_{a}\left(\beta^{-1}_{a}\left(u\right)\, u\right) = u\, \Longrightarrow
    		\alpha_{u}\left(a\right) = a, \ \beta_{a}\left(u\right) = u
    	\end{align}
    	hold, for all $a,b \in S$ and $u,v \in T$. Then, $(S,T,\alpha,\beta)$ is called a \emph{matched product system of left inverse semi-braces}.
    \end{defin}
    
   \medskip
   
   \begin{rem}\label{rem:idemp-alpha-beta}
   Note that if $e\in E(T)$, then $\alpha_{e} = \id_S$.
    In fact, there exists $u\in T$ such that $e = uu^{-1}$ and so
    \begin{align*}
        \alpha_{e}\left(a\right)
        = \alpha_{uu^{-1}}\left(a\right)
        = \alpha_{u}\alpha_{u^{-1}}\left(a\right)
        = \alpha_{u}\alpha^{-1}_{u}\left(a\right)
        = a,
    \end{align*}
    for every $a\in S$. Similarly, if $e\in \E(S)$, then $\beta_{e} = \id_T$.
   \end{rem}
    
 \noindent  Hereinafter, for every fixed pair $(a,u )\in S \times T$, we set $\bar{a}:=\alphaa{-1}{u}{\left(a\right)}$ and $\bar{u}:= \beta^{-1}_{a}\left(u\right)$, whenever it is convenient.
 Now, we introduce a preparatory lemma.

   \begin{lemma}\label{lemma-inversi}
    Let $(S,T,\alpha,\beta)$ be a matched product system of left inverse semi-braces. Then, the following
      \begin{align}
    \left(\alpha^{-1}_{u}\left(a\right)\right)^{-1}
    = \alpha^{-1}_{\beta^{-1}_{a}\left(u\right)}\left(a^{-1}\right)\label{eq:inv-alpha}\\
    \left(\beta^{-1}_{a}\left(u\right)\right)^{-1}
    = \beta^{-1}_{\alpha^{-1}_{u}\left(a\right)}\left(u^{-1}\right)
    \label{eq:inv-beta}
\end{align}
hold, for every $(a,u) \in S \times T$.
     \begin{proof} We prove the first statement, since the second one
     can be shown with the same computations where the roles of $\alpha$ and $\beta$ are reversed.
     If $(a, u)\in S \times T$, we have that
     \begin{align*}
    \alpha^{-1}_{\bar{u}}\left(a^{-1}\right)\, \bar{a}\,\alpha^{-1}_{\bar{u}}\left(a^{-1}\right)
    &= \alpha^{-1}_{\bar{u}}\left(a^{-1}\right)\alpha^{-1}_{u}\left(a\, \alpha_{\bar{u}}{\alpha^{-1}_{\bar{u}}}\left(a^{-1}\right)\right)&\mbox{by \eqref{eq:mps1}}\\
    &= \alpha^{-1}_{\bar{u}}\left(a^{-1}\right)\alpha^{-1}_{u}\left(aa^{-1}\right)\\
    &= \alpha^{-1}_{\bar{u}}\left(a^{-1}\alpha_{\beta^{-1}_{a^{-1}}\left(\bar{u}\right)}\alpha_u^{-1}\left(aa^{-1}\right)\right)&\mbox{by \eqref{eq:mps1}}\\
    &= \alpha^{-1}_{\bar{u}}\left(a^{-1}\alpha_{u}\alpha_{u}^{-1}\left(aa^{-1}\right)\right)\\
    &= \alpha^{-1}_{\bar{u}}(a^{-1}aa^{-1})\\
    &= \alpha^{-1}_{\bar{u}}(a^{-1}).
\end{align*}
Moreover, by \cref{rem:idemp-alpha-beta}, we have that
\begin{align*}
    \bar{a}\,\alpha^{-1}_{\bar{u}}\left(a^{-1}\right)\,\bar{a}
    &= \alpha^{-1}_{u}\left(a\alpha_{\bar{u}}\alpha^{-1}_{\bar{u}}\left(a^{-1}\right)\right)\bar{a}&\mbox{by \eqref{eq:mps1}}\\
    &= \alpha^{-1}_{u}\left(aa^{-1}\right)\bar{a}\\
    &= \alpha^{-1}_{u}\left(aa^{-1}\alpha_{\beta^{-1}_{aa^{-1}}\left(u\right)}\alpha^{-1}_{u}\left(a\right)\right)&\mbox{by \eqref{eq:mps1}}\\
    &= \alpha^{-1}_{u}\left(aa^{-1}\alpha_{u}\alpha^{-1}_u\left(a\right)\right)\\
    &= \alpha^{-1}_{u}(aa^{-1}a)\\
    &= \alpha^{-1}_{u}(a)
    = \bar{a}.
\end{align*}
Therefore, the condition \eqref{eq:inv-alpha} is proved.
     \end{proof}
   \end{lemma}
\medskip

\begin{theor}
		\label{th:matched-inv-semi}
		Let $\left(S,T,\alpha,\beta\right)$ be a matched product system of left inverse semi-braces. Then, $S\times T$ with respect to 
		\begin{align*}
		\left(a,u\right)+\left(b,v\right) &:=\left(a+b,u+v\right)\\
		\left(a,u\right)\left(b,v\right) &:= \left(\alphaa{}{u}{\left(\alphaa{-1}{u}{\left(a\right)}\, b\right)},\beta_{a}\left(\beta^{-1}_{a}\left(u\right)\, v\right)\right),
		\end{align*}
			for all $\left(a,u\right), \left(b,v\right) \in S \times T$, is a left inverse semi-brace, called the \emph{matched product} of $S$ and $T$ via $\alpha$ and $\beta$ and denoted by $S\bowtie T$.
		\begin{proof} Firstly, the structure $(S\times T, +)$ is trivially a semigroup. Moreover, $(S \times T, \cdot)$ is a semigroup, indeed,	set $\sigma(u)(a) = \alpha_u(a)$ and $\delta(a)\left(u\right) = \beta^{-1}_{\alpha_{u}\left(a\right)}\left(u\right)$, for all $a \in S$ and $u \in T$, it follows that $\left(S\times T, \cdot\right)$ is a Zappa product \eqref{prod-semigruppo-match} via $\sigma$ and $\delta$
		  and
		  \begin{align*}
		    \varphi: S\times T
		    \to S\times T, \, \left(a,u\right) \mapsto  \left(a, \beta_{a}\left(u\right)\right)
		  \end{align*}
		  is trivially an isomorphism from the Zappa product of $S$ and $T$ into the semigroup $\left(S\times T, \cdot\right)$. To show that $\left(S\times T,\cdot \right)$ is inverse we prove that $\left(S\times T, \cdot\right)$ is regular and that its idempotents commute. If $(a, u) \in S \times T$, by \cref{rem:idemp-alpha-beta}, we obtain
		   \begin{align*}
		   &\left(a,u\right)\left(\alpha^{-1}_{\bar{u}}\left(a^{-1}\right), \beta^{-1}_{\bar{a}}\left(u^{-1}\right)\right)\left(a,u\right)\\
		   &=  \left(a \,\alphaa{}{\bar{u}}{\alpha^{-1}_{\bar{u}}\left(a^{-1}\right) },\, u \, \beta_{\bar{a}}{ \beta^{-1}_{\bar{a}}\left(u^{-1}\right) } \right)\left(a,u\right)
    	&\mbox{by \eqref{eq:mps1}}\\
    	&=\left(a a^{-1},\, u u^{-1} \right)\left(a,u\right)\\
    	&=\left(a a^{-1} \alpha_{\beta^{-1}_{aa^{-1}}\left(u u^{-1}\right)}\left(a \right), \,
    u u^{-1} \beta_{\alpha^{-1}_{uu^{-1}}\left(a a^{-1}\right)}\left(u \right)\right)&\mbox{by \eqref{eq:mps1}}\\
    &=\left(a a^{-1} \alpha_{u u^{-1}}\left(a \right), \, u u^{-1} \beta_{aa^{-1}}\left(u \right)\right)=\left(aa^{-1}a, uu^{-1}u\right)=\left(a,u\right).
		   \end{align*}
	Furthermore, 
		  \begin{align*}
		   &\left(\alpha^{-1}_{\bar{u}}\left(a^{-1}\right), \beta^{-1}_{\bar{a}}\left(u^{-1}\right)\right)\left(a,u\right)\left(\alpha^{-1}_{\bar{u}}\left(a^{-1}\right), \beta^{-1}_{\bar{a}}\left(u^{-1}\right)\right)\\
		   &=
		   \left(\alpha^{-1}_{\bar{u}}\left(a^{-1}\right), \beta^{-1}_{\bar{a}}\left(u^{-1}\right)\right)
		   \left(a\, \alpha_{\bar{u}}\alpha^{-1}_{\bar{u}}\left(a^{-1}\right),
		   u\,
		   \beta_{\bar{a}}\beta^{-1}_{\bar{a}}\left(u^{-1}\right)
		   \right)&\mbox{by \eqref{eq:mps1}}\\
		   &=
		   \left(\alpha^{-1}_{\bar{u}}\left(a^{-1}\right), \beta^{-1}_{\bar{a}}\left(u^{-1}\right)\right)
		   \left(aa^{-1},
		   uu^{-1}\right)\\
		   &= 
		  \left(\alpha^{-1}_{\bar{u}}(a^{-1})
		  \alpha_{\beta^{-1}_{\alpha^{-1}_{\bar{u}}(a^{-1})}
		   \beta^{-1}_{\bar{a}}(u^{-1})}(aa^{-1}),
		   \, \beta^{-1}_{\bar{a}}(u^{-1})
		  \beta_{\alpha^{-1}_{\beta^{-1}_{\bar{a}}(u^{-1})}
		   \alpha^{-1}_{\bar{u}}(a^{-1})}(uu^{-1})
		  \right)&\mbox{by \eqref{eq:mps1}}
		  \end{align*}
		Now, about the first component, we have that
		   \begin{align*}
		  &\alpha^{-1}_{\bar{u}}\left(a^{-1}\right)
		  \alpha_{\beta^{-1}_{\alpha^{-1}_{\bar{u}}\left(a^{-1}\right)}
		   \beta^{-1}_{\bar{a}}\left(u^{-1}\right)}\left(aa^{-1}\right)\\
		  &= \alpha^{-1}_{\bar{u}}\left(a^{-1}\right)
		   \alpha_{\beta^{-1}_{\alpha^{-1}_{\bar{u}}\left(a^{-1}\right)}
		   \left(\beta^{-1}_{a}\left(u\right)\right)^{-1}}\left(aa^{-1}\right) &\mbox{by \eqref{eq:inv-beta}}\\
		   &= \alpha^{-1}_{\bar{u}}\left(a^{-1}\right)
		   \alpha_{\left(\beta^{-1}_{a^{-1}}\beta^{-1}_{a}\left( u\right)\right)^{-1}}\left(aa^{-1}\right)&\mbox{by \eqref{eq:inv-beta}}\\
		   &= \alpha^{-1}_{\bar{u}}\left(a^{-1}\right)
		   \alpha^{-1}_{u}\left(aa^{-1}\right)\\
		   &= \alpha^{-1}_{\bar{u}}
		   \left(a^{-1}\alpha_{\beta^{-1}_{a^{-1}}\beta^{-1}_{a}\left(u\right)}\alpha^{-1}_{u}\left(aa^{-1}\right)\right)
		   &\mbox{by \eqref{eq:mps1}}\\
		   &= \alpha^{-1}_{\bar{u}}\left(a^{-1}\right)
		   \left(a^{-1}\alpha_{u}\alpha^{-1}_{u}\left(aa^{-1}\right)\right)\\
		   &=\alpha^{-1}_{\bar{u}}
		   \left(a^{-1}aa^{-1}\right)\\
		   &=\alpha^{-1}_{\bar{u}}
		   \left(a^{-1}\right).
		   \end{align*}
	By reversing the role of $\alpha$ and $\beta$, we get that the second component is equal to $\beta^{-1}_{\bar{a}}\left(u^{-1}\right)$. Thus, 
	$\left(\alpha^{-1}_{\bar{u}}\left(a^{-1}\right), \beta^{-1}_{\bar{a}}\left(u^{-1}\right)\right)$ in an inverse of $\left(a,u\right)$. Hence, $\left(S\times T, \cdot\right)$ is a regular semigroup.
	Now, by \eqref{eq:mps-idemp}, note that $\left(a,u\right)$ is idempotent with respect to $\cdot$ if and only $a$ and $u$ are idempotent in $\left(S, \cdot\right)$ and $\left(T, \cdot\right)$, respectively. It follows that, if $\left(a,u\right)$  and $\left(b,v\right)$ are idempotents, then
	\begin{align*}
	    \left(a,u\right)\left(b,v\right) = 
	    \left(a\alpha_{\bar{u}}\left(b\right),u\beta_{\bar{a}}\left(v\right)\right)
	    = \left(a\alpha_{u}\left(b\right),u\beta_{a}\left(v\right)\right)
	    = \left(ab,uv\right)
	    = \left(ba, vu\right)
	    = \left(b\alpha_{\bar{v}}\left(a\right), v\beta_{\bar{b}}\left( u\right)\right)
	    = \left(b,v\right)\left(a,u\right).
	\end{align*}
	Therefore, the idempotents commute and so $(S \times T, \cdot)$ is an inverse semigroup. To verify that $S\times T$ is a left inverse semi-brace, i.e., \eqref{key1} holds, the computations are analogue to those in the proof of \cite[Theorem 6]{CaCoSt19}, by
	using the fact that 
		  \begin{align}\label{inverso-semigruppo-match}
		   (a,u)^{-1} = \left(\alpha^{-1}_{\beta^{-1}_{a}\left(u\right)}\left(a^{-1}\right), \beta^{-1}_{\alpha^{-1}_{u}\left(a\right)}\left(u^{-1}\right)\right).
		  \end{align}
	Therefore, the claim follows.
    \end{proof}
\end{theor}
\medskip

\begin{rem}
We observe that the map $
		    \varphi: S\times T
		    \to S\times T, \, \left(a,u\right) \mapsto  \left(a, \beta_{a}\left(u\right)\right)$
		  is an isomorphism of inverse semigroups from the Zappa product of $S$ and $T$ into the semigroup $\left(S\times T, \cdot\right)$. Indeed, clearly, $\varphi$ is an isomorphism of semigroups and 
	\begin{align*}
	    \varphi\left(\left(a,u\right)^{-1}\right)&=\varphi\left(\alpha^{-1}_{\beta^{-1}_{a}\left(u\right)}\left(a^{-1}\right), \beta^{-1}_{\alpha^{-1}_{u}\left(a\right)}\left(u^{-1}\right)\right)\\
	    &=\left( \alpha^{-1}_{\beta^{-1}_{a}\left(u\right)}\left(a^{-1}\right), \beta^{-1}_{\alpha^{-1}_{\beta^{-1}_{a}\left(u\right)}\left(a^{-1}\right)}\beta^{-1}_{\alpha^{-1}_u\left(a\right)} \left(u^{-1}\right)  \right)\\
	    &=\left(a, \beta^{-1}_a\left(u\right)\right)^{-1}\\
	    &=\left(\varphi\left(a,u\right)\right)^{-1},
	\end{align*}
    for every $(a,u) \in S \times T$. 
\end{rem}
\medskip

\noindent Similarly to \cite[Remark 5]{CaCoSt19}, by \eqref{eq:mps1}, one can check the following equalities. 
\begin{lemma}
   Let $\left(S,T,\alpha,\beta\right)$ be a matched product system of left inverse semi-braces. Then, they hold
		\begin{align}
		\label{eq:mps1old}
			\lambdaa{a}{\alphaa{}{\beta^{-1}_{a}\left(u\right)}{}} = \alphaa{}{u}{\lambdaa{\alphaa{-1}{u}{\left(a\right)}}{}}\\
		\label{eq:mps2old}
			\lambdaa{u}{\beta_{\alpha^{-1}_{u}\left(a\right)}{}} = \beta_{a}{\lambdaa{\beta^{-1}_{a}\left(u\right)}{}}
		\end{align}
for all $a\in S$ and $u \in T$.
\end{lemma}
\medskip

Now, to show that the map $r_{S\bowtie T}$ associated to a matched product $S \bowtie T$ is a solution, let us recall the notion of matched product system of solutions introduced in \cite{CCoSt20}. 
Given a solution $r_S$ on a set $S$ and a solution $r_T$ on a set $T$, if $\alpha: T \to \Sym\left(S\right)$ and $\beta: S \to \Sym\left(T\right)$ are maps, set $\alpha_u:=\alpha\left(u\right)$, for every $u \in T$, and $\beta_a:=\beta\left(a\right)$, for every $a\in S$, then the quadruple $\left(r_S,r_T, \alpha,\beta\right)$ is said to be a \emph{matched product system of solutions} if the following conditions hold
	
	{\scriptsize
		\begin{center}
			\begin{minipage}[b]{.5\textwidth}
				\vspace{-\baselineskip}
				\begin{align}\label{eq:primo}\tag{s1}
					\alpha_u\alpha_v = \alpha_{\lambda_u\left(v\right)}\alpha_{\rho_{v}\left(u\right)}
				\end{align}
			\end{minipage}%
			\hfill\hfill\hfill
			\begin{minipage}[b]{.5\textwidth}
				\vspace{-\baselineskip}
				\begin{align}\label{eq:secondo}\tag{s2}
					\beta_a\beta_b=\beta_{\lambda_a\left(b\right)}\beta_{\rho_b\left(a\right)}
				\end{align}
			\end{minipage}
		\end{center}
		\begin{center}
			\begin{minipage}[b]{.5\textwidth}
				\vspace{-\baselineskip}
				\begin{align}\label{eq:quinto}\tag{s3}
					\rho_{\alpha^{-1}_u\!\left(b\right)}\alpha^{-1}_{\beta_a\left(u\right)}\left(a\right) = \alpha^{-1}_{\beta_{\rho_b\left(a\right)}\beta^{-1}_b\left(u\right)}\rho_b\left(a\right)
				\end{align}
			\end{minipage}%
			\hfill\hfill
			\begin{minipage}[b]{.5\textwidth}
				\vspace{-\baselineskip}
				\begin{align}\label{eq:sesto}\tag{s4}
					\rho_{\beta^{-1}_a\!\left(v\right)}\beta^{-1}_{\alpha_u\left(a\right)}\left(u\right) = \beta^{-1}_{\alpha_{\rho_v\left(u\right)}\alpha^{-1}_v\left(a\right)}\rho_v\left(u\right)
				\end{align}
			\end{minipage}
		\end{center}
		\begin{center}
			\begin{minipage}[b]{.5\textwidth}
				\vspace{-\baselineskip}
				\begin{align}\label{eq:terzo}\tag{s5}
					\lambda_a\alphaa{}{\beta^{-1}_{a}\left(u\right)}{}= \alphaa{}{u}{\lambdaa{\alphaa{-1}{u}{\left(a\right)}}{}}
				\end{align}
			\end{minipage}%
			\hfill\hfill
			\begin{minipage}[b]{.5\textwidth}
				\vspace{-\baselineskip}
				\begin{align}\label{eq:quarto}\tag{s6}
					\lambdaa{u}{\betaa{}{\alphaa{-1}{u}{\left(a\right)}}{}}=\betaa{}{a}{\lambdaa{\beta^{-1}_{a}\left(u\right)}{}}
				\end{align}
			\end{minipage}
		\end{center}
	}
	\noindent for all $a,b \in S$ and $u,v \in T$.
	\medskip
	
	As shown in \cite[Theorem 2]{CCoSt20}, any matched product system of solutions determines a new solution on the set $S\times T$.
	\begin{theor}[Theorem 2, \cite{CCoSt20}]\label{th:matched-system}
	       Let $\left(r_S, \, r_T,\, \alpha,\,\beta\right)$ be a matched product system of solutions. Then, \; the map $r:S{ \times} T\times S{ \times} T \to S{ \times} T\times S{ \times} T$ defined by
		\begin{align}\label{eq:sol-match}
			&r\left(\left(a, u\right), 
			\left(b, v\right)\right) := 
			\left(\left(\alphaa{}{u}{\lambdaa{\bar{a}}{\left(b\right)}},\, \beta_a\lambdaa{\bar{u}}{\left(v\right)}\right),\ \left(\alphaa{-1}{\overline{U}}{\rhoo{\alphaa{}{\bar{u}}{\left(b\right)}}{\left(a\right)}},\,  \beta^{-1}_{\overline{A}}\rhoo{\beta_{\bar{a}}\left(v\right)}{\left(u\right)}\right) \right),
		\end{align}
		
		\noindent where we set
		\begin{center}
		   $\bar{a}:=\alphaa{-1}{u}{\left(a\right)}$, \,\,$\bar{u}:= \beta^{-1}_{a}\left(u\right)$,\,\, $A:=\alphaa{}{u}{\lambdaa{\bar{a}}{\left(b\right)}}$,\, $U:=\beta_a\lambdaa{\bar{u}}{\left(v\right)}$,\,\, $\overline{A}:=\alphaa{-1}{U}{\left(A\right)}$,\,\, $\overline{U}:= \beta^{-1}_{A}\left(U\right)$,
		\end{center}
		for all $\left(a,u\right),\left(b,v\right)\in S\times T$, is a solution. 
		This solution is called the \emph{matched product of the solutions} $r_S$ and $r_T$ (via $\alpha$ and $\beta$) and it is denoted by $r_S\bowtie r_T$.
	\end{theor}
	
    If $\left(r_{S}, r_{T}, \alpha, \beta\right)$ is a matched product system of solutions, we denote $\alphaa{-1}{u}{\left(a\right)}$ with $\bar{a}$ and $\beta^{-1}_{a}\left(u\right)$ with $\bar{u}$, when the pair $\left(a,u\right) \in S \times T$ is
	clear from the context.
	
    \medskip

	\begin{theor}\label{th:sol-match}
	Let $\left(S,T,\alpha,\beta\right)$ be a matched product system of left inverse semi-braces, $r_S$ and $r_T$ solutions on $S$ and $T$, respectively. Then, $r_{S \bowtie T}$ is a solution on the matched product $S \bowtie T$ and $r_{S \bowtie T} = r_{S}\bowtie r_{T}$.
	        	\begin{proof}
	        	Initially, we compute the components of $r_{S\bowtie T}$ and we show that they are exactly those of the matched solution of $r_S$ and $r_T$ as in \eqref{eq:sol-match}.
Specifically, if $(a, u), (b, v) \in S \times T$, we prove that
\begin{align*}
    \lambda_{\left(a,u\right)}\left(b,v\right)
    = \left(\alphaa{}{u}{\lambdaa{\bar{a}}{\left(b\right)}},\, \beta_a\lambdaa{\bar{u}}{\left(v\right)}\right)
    \qquad 
    \rho_{\left(b,v\right)}\left(a,u\right)
    = \left(\alphaa{-1}{\overline{U}}{\rhoo{\alphaa{}{\bar{u}}{\left(b\right)}}{\left(a\right)}},\,  \beta^{-1}_{\overline{A}}\rhoo{\beta_{\bar{a}}\left(v\right)}{\left(u\right)}\right).
\end{align*}
 Firstly,
\begin{align*}
    \lambda_{\left(a,u\right)}\left(b,v\right)&=\left(a,u\right)\left(\left(\alpha^{-1}_{\bar{u}}\left(a^{-1}\right),\, \beta^{-1}_{\bar{a}}\left(u^{-1}\right)\right)+\left(b,v\right)\right) &\mbox{by \eqref{inverso-semigruppo-match}}\\
    &=\left(a,u\right)\left(\left(\alpha^{-1}_{\bar{u}}\left(a^{-1}\right)+b,\, \beta^{-1}_{\bar{a}}\left(u^{-1}\right)+v\right)\right)\\
    &=\left(a\,\alpha_{\bar{u}}\left(\alpha^{-1}_{\bar{u}}\left(a^{-1}\right)+b\right), \, u\,\beta_{\bar{a}}\left(\beta^{-1}_{\bar{a}}\left(u^{-1}\right)+v\right)\right) \\
    &=\left(a\left(a^{-1}+\alpha_{\bar{u}}\left(b\right)\right), \,u\left(u^{-1}+\beta_{\bar{a}}\left(v\right)\right) \right)\\
    &=\left(\lambda_a\alpha_{\bar{u}}\left(b\right), \, \lambda_u\beta_{\bar{a}}\left(v\right)\right)\\
    &=\left(\alphaa{}{u}{\lambdaa{\bar{a}}{\left(b\right)}},\, \beta_a\lambdaa{\bar{u}}{\left(v\right)}\right) &\mbox{by \eqref{eq:mps1old} - \eqref{eq:mps2old}}
\end{align*}
Moreover, set
\begin{align*}
 x:= \alpha^{-1}_{\bar{u}}\left(a^{-1} + \alpha_{\bar{u}}\left(b\right)\right), \quad y:= \beta^{-1}_{\bar{a}}\left(u^{-1} + \beta_{\bar{a}}\left(v\right)\right),   \quad \mathcal{A}:= \alpha^{-1}_{y}
    \left(x\right), \quad
    \mathcal{U}:= \beta^{-1}_{x}
    \left(y\right),
\end{align*}
we have that
\begin{align*}
\rho_{\left(b,v\right)}\left(a,u\right)&=\left(\left(\alpha^{-1}_{\bar{u}}\left(a^{-1}\right), \beta^{-1}_{\bar{a}}\left(u^{-1}\right)\right)+\left(b,v\right)\right)^{-1}\left(b,v\right) &\mbox{by \eqref{inverso-semigruppo-match}}\\
    &=\left(\alpha^{-1}_{\bar{u}}\left(a^{-1}\right)+b, \beta^{-1}_{\bar{a}}\left(u^{-1}\right)+v\right)^{-1}\left(b,v\right)\\
    &= (x,y)^{-1}(b,v)\\
    &=\left(\alpha^{-1}_{\beta^{-1}_{x}\left(y\right)}\left(x^{-1}\right), \beta^{-1}_{\alpha^{-1}_{y}\left(x\right)}\left(y^{-1}\right)\right)\left(b,v\right) &\mbox{by \eqref{inverso-semigruppo-match}}\\
    &=\left(\mathcal{A}^{-1}, \mathcal{U}^{-1}\right)\left(b,v\right) &\mbox{by \eqref{eq:inv-alpha}-\eqref{eq:inv-beta}}\\
    &=\left(\mathcal{A}^{-1}\alpha_{\beta^{-1}_{\mathcal{A}^{-1}}\left(\mathcal{U}^{-1}\right)}\left(b\right),\, \mathcal{U}^{-1}\beta_{\alpha^{-1}_{\mathcal{U}^{-1}}\left(\mathcal{A}^{-1}\right)}\left(v\right)
    \right).
\end{align*}
Now, we focus on the first component and we show that it is equal to $\alphaa{-1}{\overline{U}}{\rhoo{\alphaa{}{\bar{u}}{\left(b\right)}}{\left(a\right)}}$. By \eqref{eq:inv-beta}, we get
\begin{align} \label{modo-A-diverso}    \mathcal{A} 
    &= \alpha^{-1}_{\bar{u}^{-1} + v}\left(x\right)
    = \alpha^{-1}_{\bar{u}\left(\bar{u}^{-1} + v\right)}\left(a^{-1} + \alpha_{\bar{u}}\left(b\right)\right)
    = \alpha^{-1}_{\lambda_{\bar{u}}\left(v\right)}\left(a^{-1} + \alpha_{\bar{u}}\left(b\right)\right)
    \end{align}
    and, similarly, by \eqref{eq:inv-alpha},
    \begin{align}
    \label{modo-U-diverso}\mathcal{U} 
   = \beta^{-1}_{\lambda_{\bar{a}}\left(b\right)}\left(u^{-1} + \beta_{\bar{a}}\left(v\right)\right).
\end{align}
Hence, 
\begin{align*} 
    \mathcal{A}^{-1}
    &= \alpha^{-1}_{\beta^{-1}_{a^{-1} + \alpha_{\bar{u}}\left(b\right)}\lambda_{\bar{u}}\left(v\right)}\left(\left(a^{-1} + \alpha_{\bar{u}}\left(b\right)\right)^{-1}\right)
    &\mbox{by \eqref{eq:inv-alpha}}\\
    &=\alpha^{-1}_{\beta^{-1}_{a^{-1}a\left(a^{-1} + \alpha_{\bar{u}}\left(b\right)\right)}\lambda_{\bar{u}}\left(v\right)}\left(\left(a^{-1} + \alpha_{\bar{u}}\left(b\right)\right)^{-1}\right)&\mbox{by \cref{rem:idemp-alpha-beta}}\\
    &=\alpha^{-1}_{\beta^{-1}_{\lambda_a\alpha_{\bar{u}}\left(b\right)}\beta_a\lambda_{\bar{u}}\left(v\right)}\left(\left(a^{-1} + \alpha_{\bar{u}}\left(b\right)\right)^{-1}\right)\\
    &=\alpha^{-1}_{\beta^{-1}_{\alpha_u\lambda_{\bar{a}}\left(b\right)}\beta_a\lambda_{\bar{u}}\left(v\right)}\left(\left(a^{-1} + \alpha_{\bar{u}}\left(b\right)\right)^{-1}\right)\\
    &=\alpha^{-1}_{\overline{U}}\left(\left(a^{-1} + \alpha_{\bar{u}}\left(b\right)\right)^{-1}\right).
\end{align*}
By \eqref{eq:mps1}, it follows that
\begin{align*}
    \mathcal{A}^{-1}\alpha_{\beta^{-1}_{\mathcal{A}^{-1}}\left(\mathcal{U}^{-1}\right)}\left(b\right)
    = \alpha^{-1}_{\bar{U}}\left(\left(a^{-1} + \alpha_{\bar{u}}\left(b\right)\right)^{-1} 
    \alpha_{\beta^{-1}_{\left(a^{-1} + \alpha_{\bar{u}}\left(b\right)\right)^{-1}}
    \left(\overline{U}\right)}
    \alpha_{\beta^{-1}_{\mathcal{A}^{-1}}\left(\mathcal{U}^{-1}\right)}\left(b\right)\right)
\end{align*}
and so, set $\mathcal{B}:= \alpha_{\beta^{-1}_{\left(a^{-1} + \alpha_{\bar{u}}\left(b\right)\right)^{-1}}
    \left(\overline{U}\right)}
    \alpha_{\beta^{-1}_{\mathcal{A}^{-1}}\left(\mathcal{U}^{-1}\right)}\left(b\right)$ the thesis reduces to show that
$$
\mathcal{B} = \alpha_{\bar{u}}\left(b \right).
$$
We obtain that the subscript of the first $\alpha$ in $\mathcal{B}$ becomes
\begin{align*}
    \beta^{-1}_{\left(a^{-1} + \alpha_{\bar{u}}\left(b\right)\right)^{-1}}
    \left(\overline{U}\right)
    &= \beta_{a^{-1} + \alpha_{\bar{u}}\left(b\right)}
    \left(\overline{U}\right)\\
    &= \beta_{a^{-1}a\left(a^{-1} + \alpha_{\bar{u}}\left(b\right)\right)}\left(\overline{U}\right)&\mbox{by \cref{rem:idemp-alpha-beta}}\\
    &= \beta^{-1}_{a}\beta_{\lambda_a\alpha_{\bar{u}\left(b\right)}}
    \beta^{-1}_{\lambda_a\alpha_{\bar{u}\left(b\right)}}\lambda_{u}\beta_{\bar{a}}\left(v\right)\\
    &= \beta^{-1}_{a}\lambda_{u}\beta_{\bar{a}}\left(v\right)&\mbox{by \eqref{eq:mps2old}}\\
    &= \lambda_{\bar{u}}\left(v\right)
\end{align*}
and so
\begin{align*}
    \mathcal{B}
    = \alpha_{\lambda_{\bar{u}}\left(v\right)}
    \alpha_{\beta^{-1}_{\mathcal{A}^{-1}}\left(\mathcal{U}^{-1}\right)}\left(b\right)
    = \alpha_{\bar{u}}\,\alpha_{\bar{u}^{-1} + v}
    \,\alpha^{-1}_{\left(\beta^{-1}_{\mathcal{A}^{-1}}\left(\mathcal{U}^{-1}\right)\right)^{-1}}\left(b\right).
\end{align*}
Thus, we show that $\left(\beta^{-1}_{\mathcal{A}^{-1}}\left(\mathcal{U}^{-1}\right)\right)^{-1} = \bar{u}^{-1}+v$. 
To this aim, note that, by \cref{rem:idemp-alpha-beta}, it holds
\begin{align}\label{eq:ultima1}
    \beta^{-1}_{\lambda_{\bar{a}}\left(b\right)}\left(u^{-1}+\beta_{\bar{a}}\left(v\right)\right)
   = \beta^{-1}_{\lambda_{\bar{a}}\left(b\right)}\beta_{\bar{a}}\left(\beta^{-1}_{\bar{a}}\left(u^{-1}\right) + v\right)
   = \beta^{-1}_{\bar{a}^{-1} + b}\left(\bar{u}^{-1} + v\right)
\end{align}
and, analogously,
\begin{align}\label{eq:ultima2}
    \alpha^{-1}_{\lambda_{\bar{u}}\left(v\right)}\left(a^{-1} + \alpha_{\bar{u}}\left(b\right)\right)
  = \alpha^{-1}_{\bar{u}^{-1} + v}\left(\bar{a}^{-1} + b\right).
\end{align}
We compute 
\begin{align*}
 &\left(\beta^{-1}_{\mathcal{A}^{-1}}\left(\mathcal{U}^{-1}\right)\right)^{-1}\\
 &=\beta^{-1}_{\left(\alpha^{-1}_{\lambda_{\bar{u}}\left(v\right)}\left(a^{-1}+\alpha_{\bar{u}}\left(b\right)\right)\right)^{-1}}  \left(\beta^{-1}_{\lambda_{\bar{a}}\left(b\right)}\left(u^{-1}+\beta_{\bar{a}}\left(v\right)\right)\right)^{-1} &\mbox{by \eqref{modo-A-diverso}-\eqref{modo-U-diverso}}\\
 &=\beta^{-1}_{\alpha^{-1}_{\left(\beta^{-1}_{\lambda_{\bar{a}}\left(b\right)}\left(u^{-1}+\beta_{\bar{a}}\left(v\right)\right)\right)^{-1}}\left(\alpha^{-1}_{\lambda_{\bar{u}}\left(v\right)}\left(a^{-1}+\alpha_{\bar{u}}\left(b\right)\right)\right)^{-1}}\beta^{-1}_{\lambda_{\bar{a}}\left(b\right)}\left(u^{-1}+\beta_{\bar{a}}\left(v\right)\right) &\mbox{by \eqref{inverso-semigruppo-match}}\\
 &=\beta^{-1}_{\alpha^{-1}_{\left(\beta^{-1}_{\bar{a}^{-1} + b}\left(\bar{u}^{-1} + v\right)\right)^{-1}}\left(\alpha^{-1}_{\bar{u}^{-1} + v}\left(\bar{a}^{-1} + b\right)\right)^{-1}}\beta^{-1}_{\bar{a}^{-1} + b}\left(\bar{u}^{-1} + v\right) &\mbox{\eqref{eq:ultima1}-\eqref{eq:ultima2}}\\
 &=\beta^{-1}_{\left(\bar{a}^{-1} + b\right)^{-1}}\beta^{-1}_{\bar{a}^{-1} + b}\left(\bar{u}^{-1} + v\right)&\mbox{by \eqref{inverso-semigruppo-match}}\\
 &=\bar{u}^{-1} + v,
\end{align*} 
which proves what stated before. By reversing the role of $\alpha$ and $\beta$, we obtain with similar computations that the second component of $\rho_{(b,v)}(a,u)$ is equal to $\beta^{-1}_{\overline{A}}\rhoo{\beta_{\bar{a}}\left(v\right)}{\left(u\right)}$.\\
Now, to get the claim, it remains to be proven that $\left(r_S,r_T, \alpha,\beta\right)$ is matched product system of solutions, by verifying only \eqref{eq:primo}, \eqref{eq:quinto}, and \eqref{eq:terzo}, since the the other ones can be obtained similarly.
Initially, note that, by \cref{rem:idemp-alpha-beta}
\begin{align*}
    \alpha_{\lambda_{u}\left(v\right)}
    \alpha_{\rho_{v}\left(u\right)}
    &= \alpha_{\lambda_{u}\left(v\right)\rho_{v}\left(u\right)}
    = \alpha_{u\left(u^{-1} + v\right)\left(u^{-1} + v\right)^{-1}v}
    = \alpha_{u}
    \alpha_{\left(u^{-1} + v\right)\left(u^{-1} + v\right)^{-1}}\alpha_{v}
    = \alpha_{u}\alpha_{v} 
    = \alpha_{uv},
\end{align*}
hence \eqref{eq:primo} holds. Moreover,
\begin{align*}
    \rho_{\alpha^{-1}_u\!\left(b\right)}\alpha^{-1}_{\beta_a\left(u\right)}\left(a\right) &=\rho_{\alpha^{-1}_u\!\left(b\right)}\left(\left(\alpha_u^{-1}\left(a^{-1}\right)\right)^{-1}\right) &\mbox{by \eqref{eq:inv-alpha}}\\
    &=\left(\alpha_u^{-1}\left(a^{-1}\right)+\alpha_u^{-1}\left(b\right)\right)^{-1}\alpha_u^{-1}\left(b\right)\\
    &=\left(\alpha_u^{-1}\left(a^{-1}+b\right)\right)^{-1}\alpha_u^{-1}\left(b\right)\\
    &=\left(\left(\alpha^{-1}_u\left(b\right)\right)^{-1}\alpha^{-1}_u\left(a^{-1} + b\right)\right)^{-1}\\
    &= \left(\alpha^{-1}_{\beta^{-1}_b\left(u\right)}\left(b^{-1}\right)\alpha^{-1}_u\left(a^{-1} + b\right)\right)^{-1}\\
    &= \left(\alpha^{-1}_{\beta^{-1}_b\left(u\right)}\left(b^{-1}\alpha^{-1}_{\beta^{-1}_{b^{-1}}\beta^{-1}_{b}\left(u\right)}\alpha^{-1}_{u}\left(a^{-1} + b\right)\right)
    \right)^{-1} &\mbox{by \eqref{eq:mps1}}\\
    &= \left(\alpha^{-1}_{\beta^{-1}_b\left(u\right)}\left(b^{-1}\left(a^{-1} + b\right)\right)
    \right)^{-1}\\
    &=\left(\alpha^{-1}_{\beta^{-1}_b\left(u\right)}\left(\rho_b\left(a\right)\right)^{-1}\right)^{-1}\\
    &=\alpha^{-1}_{\beta_{\rho_b\left(a\right)}\beta^{-1}_b\left(u\right)}\rho_b\left(a\right)&\mbox{by \eqref{eq:inv-alpha}}
\end{align*}
hence \eqref{eq:quinto} is satisfied.
Finally, it is clear that \eqref{eq:terzo} coincides with \eqref{eq:mps1old}.\\
Therefore, by \cref{th:matched-system}, the claim follows.
\end{proof}
\end{theor}

\bigskip

In the next, we focus on a particular case of the previous construction, the semidirect product of two inverse semi-braces. From now on, given a matched product system $\left(S, T, \alpha, \beta\right)$, we consider $\beta_a = \id_T$, for every $a\in S$. In this way, the inverse semigroup $\left(S\times T, \cdot\right)$ is exactly the semidirect product of the inverse semigroups $\left(S,\cdot \right)$ and $\left(T,\cdot \right)$ via $\alpha$, in the sense of
\cite{Ni83} and \cite{Pr86}. 
Let us note that this semidirect product is a particular case of the Zappa product of two semigroups with $^{u}a = \sigma\left(u\right)\left(a\right) = \alpha_u\left(a\right)$  and $u^a = \beta_a\left(u\right) = u$, for all $a\in S$ and $u\in T$. Specifically, the multiplication $\cdot$ on $S \times T$ is given by
\begin{align*}
(a,u) (b,v) = \left(a\, ^{u}b, \, uv \right),
\end{align*}
for all $(a,u), (b,v)\in S\times T$. 
\medskip

Let us recall the result essentially contained in \cite[Theorem 6]{Pr86} that, in contrast to \cref{inverse-semigroup-wazzan}, is a
characterization to obtain an inverse semigroup.
\begin{theor}\label{th:char-semiprod}
	Let $S, T$ be semigroups and $\sigma:T\to \End(S)$ a homomorphism. Then, the semidirect product of $S$ and $T$ via $\sigma$ is an inverse semigroup if and only if they hold
	\begin{enumerate}
		\item $S$ and $T$ are inverse semigroups;
		\item $\sigma(T)\subseteq \Aut(S)$.
	\end{enumerate}
\end{theor}

\noindent Let us note that condition $2.$ in the previous theorem is equivalent to the property 
\begin{align}\label{aut-idemp}
 ^{e}a = a,   
\end{align}
for all $e\in E(T)$ and $a\in S$. Such a condition derives also from \cref{rem:idemp-alpha-beta}.\\
As a consequence of \eqref{aut-idemp}, $\sigma$ is a homomorphism of inverse semigroups, namely, in addition, it holds $\sigma\left(u^{-1}\right) = \sigma\left({u}\right)^{-1}$, for every $u \in T$.
Moreover, one can check that
\begin{align}\label{inverso}
  (a,u)^{-1}
  = \left(^{u^{-1}}{a^{-1}}, u^{-1}\right),
\end{align}
for every $(a,u)\in S \times T$, which follows also from \cref{lemma-inversi}.

\bigskip

The following is an example of semidirect product of two Clifford semigroups which is not a Clifford semigroup.
Specifically, this shows that left inverse semi-braces can be obtained also starting from two particular generalized left semi-braces, as one can concretely see later in \cref{ex:gen-inve}. 
\begin{ex}\label{ex:semidirect2}
Considered the set $X:=\{1,x,y\}$, let
$S$ be the upper semilattice on $X$ with join $1$ and $T$ the commutative inverse monoid with identity $1$ for which $xx = yy= x$  and $xy = y$. If $\tau$ is the automorphism of $S$ given by the transposition $\tau := (x\, y)$, then the map $\sigma:T\to \Aut(S)$ given by $\sigma(1) = \sigma(x) = \id_S$ and $\sigma(y) = \tau$, 
is a homomorphism from $T$ into $\Aut(S)$. Therefore, by \cref{th:char-semiprod}, it follows that the semidirect product of $S$ and $T$ via $\sigma$ is an inverse semigroup. Observe that such a semigroup is not a Clifford semigroup. Indeed, since by \eqref{inverso} it holds that
$(x,y)^{-1}=(y,y)$, we have
$(x,y)(x,y)^{-1} 
= (x,x)$, but $(x,y)^{-1}(x,y)= (y,x)$.
\end{ex}

\medskip

\noindent The following result is a consequence of \cref{th:matched-inv-semi}.
\begin{cor}\label{th:semi-inverse-semi-b}
	Let $S,T$ be left inverse semi-braces and $\sigma:T\to \Aut(S)$ a homomorphism from $(T, \cdot)$ into the automorphism group of the left inverse semi-brace $S$. Then, the structure $(S\times T, +, \cdot)$ where
	\begin{align*}
	(a,u)+(b,v)&:=(a + b,\,u + v )\\
	(a,u)\,(b,v)&:= (a\, ^{u}b,\, uv ),
	\end{align*}
	for all $(a,u), (b,v)\in S\times T$, is a left inverse semi-brace. We call such an inverse semi-brace the \emph{semidirect product of $S$ and $T$ via $\sigma$} and we denote it by $S\rtimes_{\sigma}T$.
\end{cor}

\medskip

\begin{ex}\label{ex:semi-brace-inverso-1}
Let $X:=\{1,x,y\}$ be a set, $S$ the cyclic group on $X$, and $T$ the commutative inverse monoid on $X$ in \cref{ex:semidirect2}. Considered $\iota:S\to S$ the automorphism of $S$ defined by $\iota(a) = -a$, for every $a\in S$, we obtain that the map $\sigma:T\to \Aut(S)$ given by 
$\sigma(1) = \sigma(x) = \id_S$, and $\sigma(y) = \iota$
is a homomorphism from $T$ into $\Aut(S)$.
Therefore, by \cref{th:char-semiprod}, it follows that the semidirect product of $S$ and $T$ via $\sigma$ is an inverse semigroup (which is a Clifford semigroup).
Now, set $a + b = a$, for all $a,b\in S$, and
$u + v = uv$, for all $u,v\in T$, then $S$ is a left semi-brace and, as seen in $1.$ of \cref{ex:first-examples}, $T$ is a generalized left semi-brace. In addition, the map $\sigma$  is a homomorphism from $T$ into the automorphism of the inverse semi-brace $S$. Therefore, by \cref{th:semi-inverse-semi-b},
$S\times T$ endowed with the following operations
\begin{align*}
	(a,u)+(b,v) =(a,\,uv )\qquad 
	(a,u)\,(b,v) = (a\, ^{u}b,\, uv ),
	\end{align*}
	for all $(a,u), (b,v)\in S\times T$, is a left inverse semi-brace that is the semidirect product of $S$ and $T$ via $\sigma$. In particular, $S \rtimes_{\sigma} T$ is a generalized left semi-brace.
\end{ex}
\medskip

The semidirect product of two generalized semi-braces can be an inverse semi-brace which is not a generalized semi-brace.
\begin{ex} \label{ex:semi-brace-inverso-2} Considered the inverse semigroups $(S, \cdot)$ and $(T, \cdot)$ in the
\cref{ex:semidirect2}, set $a + b = b$, for all $a,b\in S$, and
$u + v = uv$, for all $u,v\in T$, then $S$ is a left semi-brace and $T$ is a generalized semi-brace. 
Furthermore, the map $\sigma:T\to \Aut(S)$ defined by $\sigma(1) = \sigma(x) = \id_S$ and $\sigma(y) = \tau$, is a homomorphism from $T$ into the automorphism of the inverse semi-brace $S$. Therefore, by \cref{th:semi-inverse-semi-b},
$S\times T$ endowed with the following operations
\begin{align*}
	(a,u)+(b,v) =(b,\,uv )\qquad
	(a,u)\,(b,v) = (a\, ^{u}b,\, uv ),
	\end{align*}
	for all $(a,u), (b,v)\in S\times T$, is a left inverse semi-brace that is the semidirect product of $S$ and $T$ via $\sigma$. 
\end{ex}

\medskip
As a consequence of \cref{th:sol-match}, given two solutions $r_S$ and $r_T$ on two left inverse semi-braces $S$ and $T$, respectively, the map $r_B$ associated to a semidirect product $B:=S \rtimes_{\sigma}T$ via a homomorphism $\sigma: T\to \Aut(S)$ is still a solution.

\begin{cor}\label{theor:soluzione-semidiretto}
Let $S$, $T$ be left inverse semi-braces, $r_S$ and $r_T$ solutions on $S$ and $T$, respectively, and $\sigma:T\to \Aut(S)$ a homomorphism from $T$ into the automorphism group of the inverse semi-brace $S$. Then, the map $r_B$ associated to the semidirect product $B:=S \rtimes_{\sigma} T$, given by
\begin{align*}
   r_B\left(\left(a,u\right), \left(b,v\right)\right) 
   = \left(\left({}^{u}\lambda_{{}^{u^{-1}}{a}}\left(b\right),\lambda_u\left(v\right)\right),\,
   \left(^{\lambda_u\left(v\right)^{-1}}\rho_{_{{}^u b}}\left(a\right), \rho_v\left(u\right)\right)\right),
\end{align*}
for all $(a,u),(b,v)\in S\times T$, is a solution. In particular, such a solution $r_B$ is actually the semidirect product of the solutions $r_S$ and $r_T$.
\end{cor}

\medskip

\begin{ex}\label{ex:semiprod-trivial}
Let $S$, $T$ be the trivial left inverse semi-braces in \cref{trivial-semibrace} with $(S,+)$ a left zero semigroup and $(T,+)$ a right zero semigroup, and $\sigma: T \to \Aut\left(S\right)$ a homomorphism from $\left(T,\cdot\right)$  into the automorphism group of the left inverse semi-brace $S$. Then, by \cref{th:semi-inverse-semi-b}, $B:=S \rtimes_{\sigma} T$ is a left inverse semi-brace and, by \cref{theor:soluzione-semidiretto}, the map $r_B$ associated to $B$ given by
\begin{align*}
    r_B\left(\left(a,u\right), \left(b,v\right)\right)=\left(\left(aa^{-1}, uv\right), \left(^{v^{-1}}\left(a \, ^{u}b\right), v^{-1}v\right)\right)
\end{align*}
is a solution. Specifically, it is exactly the semidirect product of the solution $r_S$ and $r_T$ in \cref{ex:sol} given by $r_S(a,b)=(aa^{-1},ab)$ and $r_T(u,v)=(uv,v^{-1}v)$, respectively. Moreover, $r_B$ is idempotent, consistently with \cite[Corollary 5]{CCoSt20}.
\end{ex}

\medskip

To analyze the next examples, 
let us recall the notions of \emph{index} and \emph{period} of a solution $r$ introduced in \cite{CaCoSt19} that are respectively defined as 
    \begin{align*}
        &\indd{\left(r\right)}
        :=\min\left\{\left.j \,\right|\, j\in\mathbb{N}_0, \, \exists \, l\in \mathbb{N}\ r^j=r^l , \ j\neq l\right\},\\
        &\perr{\left(r\right)}
        :=\min\left\{\left.k\,\right| \, k\in\mathbb{N}, \, r^{\indd{\left(r\right)}+k}=r^{\indd{\left(r\right)}}\right\}.
    \end{align*}
    These definitions of the index and the order are slightly different from the classical ones (cf. \cite[p. 10]{Ho95}). This choice is functional to distinguish bijective solutions, having index $0$, from non-bijective ones, having index a positive integer.  

\begin{ex}
Let $B:=S \rtimes_{\sigma}T$ be the left inverse semi-brace in \cref{ex:semi-brace-inverso-1}. Then, by \eqref{aut-idemp}, the map $r_B$ associated to $B$ is given by
\begin{align*}
    r_B\left(\left(a,u\right),\left(b,v\right) \right)=\left(\left(1,uu^{-1}v \right), \left(^{v^{-1}} \left(a\, ^u b\right), uv^{-1}v\right)\right),
\end{align*}
and, by \cref{theor:soluzione-semidiretto}, it is a solution. Specifically, $r_B$ is the semidirect product of the solutions $r_S(a,b)=(1,ab)$ on $S$ (cf. \cref{ex:sol}) and $r_T(u,v)=(uu^{-1}v,uv^{-1}v)$ on $T$ (see \cref{ex-sol-semibrace-clifford}), respectively. 
Moreover, by \cite[Proposition 10]{CaCoSt19}, such a solution $r_B$ is cubic with $\ind(r_B)=1$ and $\per(r_B)=2$.
\end{ex}

\medskip 

\begin{ex}\label{ex:gen-inve} 
Let $B:=S \rtimes_{\sigma}T$ be the left inverse semi-brace in
\cref{ex:semi-brace-inverso-2}. Then, by \eqref{aut-idemp}, the map $r_B$ associated to $B$ is given by
\begin{align*}
    r_B\left(\left(a,u\right),\left(b,v\right) \right)=\left(\left(a \, ^u b,uu^{-1}v \right), \left(^{v^{-1}u} \left( b^{-1}b\right), uv^{-1}v\right)\right),
\end{align*}
and, by \cref{theor:soluzione-semidiretto}, it is a solution. In particular, $r_B$ is the semidirect product of the solutions $r_S(a,b)=(ab, b^{-1}b)$ on $S$ (cf. \cref{ex:sol}) and $r_T(u,v)=(uu^{-1}v,uv^{-1}v)$ on $T$ (see \cref{ex-sol-semibrace-clifford}), respectively. Furthermore, by \cite[Proposition 10]{CaCoSt19}, such a solution $r_B$ is cubic with $\ind(r_B)=1$ and $\per(r_B)=2$.
\end{ex}

\bigskip

\section{The double semidirect product of left inverse semi-braces}
This section is devoted to present a new construction of left inverse semi-braces which includes the semidirect product in \cref{th:semi-inverse-semi-b}, that is the double semidirect product. 
In particular, we show that under mild assumptions the map associated to the double semidirect product of arbitrary left semi-braces is a solution. 

\medskip

\begin{theor}\label{th:double-inv-semi}
  		Let $S$ and $T$ be two left inverse semi-braces,  $\sigma: T\to \Aut\left(S\right)$ a homomorphism from $\left(T,\cdot\right)$ into the automorphism group of the left inverse semi-brace $S$, with ${}^u a:= \sigma(u)(a)$, for all $a \in S$ and $u \in T$, and $\delta:S\to \End\left(T\right)$ an anti-homomorphism from $\left(S, +\right)$ into the endomorphism semigroup of $\left(T, +\right)$, with $u^a:=\delta(a)(u)$, for all $a \in S$ and $u \in T$. If the condition
	\begin{align}\label{eq:semibrace-sigma-delta}
        \left(uv\right)^{\lambda_{a}\left({}^ub\right)}
		+ u\left(\left(u^{-1}\right)^b
		+ w\right)
		= u \left(v^b +w\right)
	\end{align}
		holds, for all $a,b \in S$ and $u,v, w\in T$, then $S \times T$ with respect to
		\begin{align*}
		 \left(a,u\right)+\left(b,v\right)&:=\left(a+b, \, u^b+v\right)\\
		 \left(a,u\right)\left(b,v\right)&:=\left(a \ ^{u}{b}, \, uv\right),
		\end{align*}
		for all $\left(a,u\right), \left(b,v\right) \in S \times T$,
		is a left inverse semi-brace. We call such a left inverse semi-brace the \emph{double semidirect product} of $S$ and $T$ via $\sigma$ and $\delta$.
		\begin{proof}
		 At first, note that the structure $\left(S\times T, +\right)$ is a semigroup since it is exactly the semidirect product of the semigroup $\left(S,+\right)$ and $\left(T,+\right)$ via $\delta$.
		Moreover, by \cref{th:char-semiprod}, $(S \times T, \cdot)$ is an inverse semigroup. Thus, it only remains to prove that \eqref{key1} is satisfied.
		If $(a,u), (b,v), (c, w) \in S \times T$, we obtain
			 \begin{align*}
	 \left(a,u\right)&\left(\left(b,v\right)+\left(c,w\right)\right)= \left(a,u\right)\left(
	b+c, v^c+w\right)=\left( a \ {}^{u}\left(b+c\right), u \left(v^c+w\right)\right)
	\end{align*}
	and
		\begin{align*}
		   \left(a,u\right)\left(b,v\right)
		    + &\left(a,u\right)\left(\left(a,u\right)^{-1} + \left(c,w\right)\right)\\
		    &=\left(a \ ^{u}{b}, \, uv\right)+\left(a, u \right)\left(\left({}^{u^{-1}}a^{-1}, u^{-1}\right) + \left(c, w\right)\right)&\mbox{by \eqref{inverso}}\\
		    &=\left(a \ ^{u}{b}, \, uv\right)+\left(a,u\right)\left({}^{u^{-1}}a^{-1}+c,\, \left(u^{-1}\right)^c+w\right)\\
		    &=\left(a \ ^{u}{b}, \, uv\right)+\left(a \ ^{u}{\left(^{u^{-1}}{a^{-1}}+c\right)}, \, u\left(\left(u^{-1}\right)^c+w\right)\right)\\
		    &= \left(a \ ^{u}{b}, \, uv\right) + \left(a \left(a^{-1} + \ ^{u}c\right), \, u\left(\left(u^{-1}\right)^c+w\right)\right)&\mbox{by \eqref{aut-idemp}}
		    \\
		    &= \left(a \ ^{u}{b} + \lambda_{a}\left(^{u}c\right), \, \left(uv\right)^{\lambda_a\left(^{u} c\right)}+ u\left(\left(u^{-1}\right)^c+w\right)\right)\\
		     &=\left(a \left( ^{u}{b} +  ^{u}{c}\right), \, \left(uv\right)^{\lambda_a\left(^{u} c\right)}+ u\left(\left(u^{-1}\right)^c+w\right)\right),
		    \end{align*}
		hence, by \eqref{eq:semibrace-sigma-delta} the claim follows.
	\end{proof}
\end{theor}

\medskip

We specialize the \cref{th:double-inv-semi} for (left cancellative) left semi-braces and skew left braces.
\begin{cor}\label{prop:weak-left-canc-semi}
Let $S,T$ be (left cancellative) left semi-braces, $\sigma: T\to \Aut\left(S\right)$ a homomorphism from $\left(T,\cdot\right)$ into the automorphism group of $(S, +, \cdot)$, and $\delta:S\to \End\left(T\right)$ an anti-homomorphism from $\left(S, +\right)$ into the automorphism group of $\left(T, +\right)$ satisfying \eqref{eq:semibrace-sigma-delta}. Then, the double semidirect product of $S$ and $T$ via $\sigma$ and $\delta$ is a (left cancellative) semi-bracee. 
\end{cor}

 We remind that a skew left brace is a triple $(S, +, \cdot)$ where $(S, +)$ and $(S, \cdot)$ are groups, see \cite{GuVe17}.  
 Let us note that given two skew left braces $S$ and $T$, the double semidirect product of $S$ and $T$ is not necessarily a skew left braces, too. 
Indeed, it is a consequence of the fact that the semidirect product of two groups $(S, +)$ and $(T, +)$ with $\delta\left(T\right)\not\subseteq \Aut\left(T\right)$ in general is not a group, see \cite{Pr86}.

 \begin{cor}\label{prop:weak-skew}
    Let $S,T$ be skew left braces, $\sigma: T\to \Aut\left(S\right)$ a homomorphism from $\left(T,\cdot\right)$ into the automorphism group of $(S, +, \cdot)$, and $\delta:S\to \Aut\left(T\right)$ an anti-homomorphism from $\left(S, +\right)$ into the automorphism group of $\left(T, +\right)$ satisfying \eqref{eq:semibrace-sigma-delta}. Then, the double semidirect product of $S$ and $T$ via $\sigma$ and $\delta$ is a skew left brace. 
\end{cor}

\medskip

Now, our aim is to study the solution associated to any double semidirect product. For convenience, to calculate such a map, in the following lemma we rewrite conditions
\eqref{eq:mps1old}-\eqref{eq:mps2old} in the exponential notation.
\begin{lemma}\label{lem:azioni}
    Let $S,T$ be a left inverse semi-brace and $\sigma:T\to \Aut(S)$ a homomorphism from $(T, \cdot)$ into the automorphism of the inverse semi-brace $S$. Then, the following properties are satisfied:
    \begin{enumerate}
        \item ${}^{u}\lambda_{a}(b) = \lambda_{{}^{u}a}({}^{u}b)$
        \item ${}^{u}\rho_{b}(a) 
        = \rho_{{}^{u}b}({}^{u}a)$,
    \end{enumerate}
    for all $a,b\in S$ and $u\in T$.
\end{lemma}
\medskip

To simplify our computations concerning the map associated to a double semidirect product, hereinfater we use the notation
\begin{align*}
    \Omega_{u,v}^{a}
    := \left(u^{-1}\right)^a+v,
\end{align*}
for all $a\in S$, $u,v\in T$. 
\begin{prop}\label{r_weakproduct}
    Let $S$, $T$ be left inverse semi-braces and $B$ the double semidirect product of $S$ and $T$ via $\sigma$ and $\delta$. Then, the map $r_B$ associated to $B$ is given by
    \begin{align*}
        r_B\left(\left(a,u\right), \left(b,v\right)\right)
        =
        \left(\left(\lambda_a\left(^{u}{b}\right), u\,\Omega_{u,v}^{b}\right),
        \left(^{\left(\Omega_{u,v}^{b}\right)^{-1}u^{-1}}{\rho_{^u b}\left(a\right)}, \, \left(\Omega_{u,v}^{b}\right)^{-1}v\right)\right),
    \end{align*}
    for all $\left(a,u\right), \left(b,v\right)\in S\times T$.
\begin{proof}
Let us compute the components of the map $r_B$. If $a,b \in S$ and $u, v \in T$, by \eqref{inverso}, we have that
\begin{align*}
\lambda_{\left(a,u\right)}\left(b,v\right)=\left(a \, ^{u}{\left(^{u^{-1}}{a^{-1}}+b\right)}, 
\, u \,\Omega_{u,v}^{b}\right)=\left(a \, \left(a^{-1}+^{u}{b}\right), \, u \,\Omega_{u,v}^{b}\right)
= \left(\lambda_a\left(^{u}{b}\right), u\, \Omega_{u,v}^{b}\right)
\end{align*}
and, by \eqref{inverso} and \cref{lem:azioni},
\begin{align*}
    \rho_{\left(b,v\right)}\left(a,u\right)
    &=\left({}^{\left(\Omega_{u,v}^{b}\right)^{-1}}{\left(^{u^{-1}}{a^{-1}}+b\right)^{-1}}  \ {}^{\left(\Omega_{u,v}^{b}\right)^{-1}}{b}, \, \left(\Omega_{u,v}^{b}\right)^{-1}v\right)\\
    &=\left( ^{\left(\Omega_{u,v}^{b}\right)^{-1}}{\left(\left(^{u^{-1}}{a^{-1}}+b\right)^{-1}b\right)}, \, \left(\Omega_{u,v}^{b}\right)^{-1}v\right)\\
    &=\left( ^{\left(\Omega_{u,v}^{b}\right)^{-1}}{\rho_b\left(\left(^{u^{-1}}{a^{-1}}\right)^{-1}\right)}, \, \left(\Omega_{u,v}^{b}\right)^{-1}v\right)\\
    &=\left(^{\left(\Omega_{u,v}^{b}\right)^{-1}}{\rho_b\left(^{u^{-1}}{a}\right)}  ,  \, \left(\Omega_{u,v}^{b}\right)^{-1}v\right)\\
    &=\left(^{\left(\Omega_{u,v}^{b}\right)^{-1}u^{-1}}{\rho_{^u b}\left(a\right)}, \left(\Omega_{u,v}^{b}\right)^{-1}v\right).
\end{align*}
Therefore, the claim follows.
\end{proof}
\end{prop}

\noindent Note that the second component of \, $\lambda_{(a,u)}(b,v)$ \, and \, $\rho_{(b,v)}(a,u)$ \, can be written also as
\begin{align*}
u \,\Omega_{u,v}^{b}
=
u \left(u^{-1}\right)^b+\lambda_u\left(v\right)
\qquad 
\left(\Omega_{u,v}^{b}\right)^{-1}v =\rho_v\left(\left(\left(u^{-1}\right)^b\right)^{-1}\right),
\end{align*}
for all $(a,u), (b,v) \in S \times T$. We will use the two forms at the convenience of our computations.
\medskip

Let us note that the maps associated to the double semidirect product of two skew left braces and that of two left cancellative left braces, applying \cref{prop:weak-skew} and  \cref{prop:weak-left-canc-semi}, are automatically solutions.
Since in general this does not happen, now we focus on sufficient conditions that allow for constructing new solutions through the double semidirect product of left semi-braces. We highlight that the  condition $1.$ in the following theorem can be generalized as one can see later in \cref{rem:altracondizione}, but our choice was dictated by the need to find examples in an easier way. 

\begin{theor}\label{th:weak-asym-groups}
      Let $S$, $T$ be left semi-braces and $B$ the double semidirect product of $S$ and $T$ via $\sigma$ and $\delta$. If $r_S$ and $r_T$ are solutions associated to $S$ and $T$, respectively, and the following are satisfied  
      \begin{enumerate}
          \item $\left(u^{1}\right)^a = u^a$,
          \item $1^a + u = 1 + u$,
      \end{enumerate}
      for all $a\in S$ and $u\in T$, then the map $r_B$ associated to $B$ is a solution.
      \end{theor}
      \begin{proof}
      By \cref{th-generalized-sol}, the map $r_B$ associated to $B$ is a solution if and only if the condition \eqref{eq:condsolution} is satisfied. If $(a,u), (b, v), (c,w) \in S \times T$, then we get
       \begin{align*}
           &\lambda_{\left(b,v\right)}\left(c,w\right)
           \left(\left(1,1\right) + \rho_{\left(c,w\right)}\left(b,v\right)\right)\\
           &=\lambda_{\left(b,v\right)}\left(c,w\right)\left(\left(1,1\right) +\left(^{\left(\Omega_{v,w}^{c}\right)^{-1}} \rho_c\left(^{v^{-1}} b\right)\right), \left(\Omega_{v,w}^{c}\right)^{-1}w\right)\\
           &=\left(\lambda_b\left(^v c\right), \, v\, \Omega_{v,w}^{c}\right) \left(1+^{\left(\Omega_{v,w}^{c}\right)^{-1}} \rho_c\left(^{v^{-1}} b\right), 1^{^{\left(\Omega_{v,w}^{c}\right)^{-1}} \rho_c\left(^{v^{-1}} b\right)}+\left(\Omega_{v,w}^{c}\right)^{-1}w\right)\\
          &=\left(\lambda_b\left(^v c\right)\, \, ^{v \,\Omega_{v,w}^{c}}\left(1+^{\left(\Omega_{v,w}^{c}\right)^{-1}} \rho_c\left(^{v^{-1}} b\right)\right),  v \, \Omega_{v,w}^{c} \left( 1^{^{\left(\Omega_{v,w}^{c}\right)^{-1}} \rho_{c}\left(^{v^{-1}} b\right)}+\left(\Omega_{v,w}^{c}\right)^{-1}w \right)\right)\\
          &=\left(\lambda_b\left(^v c\right)\left(1 +\, ^v \, \left(\rho_c\left(^{v^{-1}} b\right)\right)\right),\,  v\, \Omega_{v,w}^{c} \left( 1^{^{\left(\Omega_{v,w}^{c}\right)^{-1}} \rho_c\left(^{v^{-1}} b\right)}+\left(\Omega_{v,w}^{c}\right)^{-1}w \right)\right),
\end{align*}
hence
\begin{align*}
    &\left(a,u\right)
           +\lambda_{\left(b,v\right)}\left(c,w\right)
           \left(\left(1,1\right) + \rho_{\left(c,w\right)}\left(b,v\right)\right)\\
           &= \left(a + \lambda_b\left(^v c\right)\left(1 + \rho_{{}^v c}\left(b\right)\right),  u^{\lambda_b\left(^v c\right)\left(1 + \rho_{{}^v c}\left(b\right)\right)} + v\,\Omega_{v,w}^{c}\left(
           1^{ ^  {\left(\Omega_{v,w}^{c}\right)^{-1}} \rho_c\left({}^{v^{-1}}b\right)}
+ \left(\Omega_{v,w}^{c}\right)^{-1}w\right)\right)
\end{align*}
        and
      \begin{align*}
          \left(a,u\right)+\left(b,v\right)\left(\left(1,1\right)+\left(c,w\right)\right)&=\left(a,u\right)+\left(b,v\right)\left(1+c,\, 1^c+w\right)\\
          &=\left(a,u\right)+\left(b \, ^v \left(1+c\right),\, v\left(1^c+w\right)\right)\\
          &=\left(a+b\left(1+^v c\right),\, u^{b\left(1+\, ^vc\right)}+v\left(1^c+w\right)\right).
      \end{align*}
      Note that, by \eqref{eq:condsolution} in \cref{th-generalized-sol}, since $r_S$ is a solution, the first components are equal.
      About the second components, we have that
      \begin{align*}
          u^{\lambda_b\left(^v c\right)\left(1 + \rho_{{}^v c}\left(b\right)\right)}
          &\underset{1.}{=} \left(u^1\right)^{\lambda_b\left(^v c\right)\left(1 + \rho_{{}^v c}\left(b\right)\right)} 
          = u^{1 + \lambda_b\left(^v c\right)\left(1 + \rho_{{}^v c}\left(b\right)\right)}
          \underset{\eqref{eq:condsolution}}{=} u^{1 + b\left(1 + ^v c\right)}= \left(u^{1}\right)^{b\left(1 + ^v c\right)}\underset{1.}{=} u^{b\left(1 + ^v c\right)}.
      \end{align*}
      In addition, 
      \begin{align*}
          &v\, \Omega_{v,w}^{c} \left(
           1^{ ^  {\left(\Omega_{v,w}^{c}\right)^{-1}} \rho_c\left({}^{v^{-1}}b\right)}
           + \left(\Omega_{v,w}^{c}\right)^{-1}w\right)\\
           &= v \,\Omega_{v,w}^{c} \left(
           1+ \left(\Omega_{v,w}^{c}\right)^{-1}w\right)&\text{by $2.$}\\
           &= v(\left(v^{-1}\right)^c\left((\left(v^{-1}\right)^c\right)^{-1}\Omega_{v,w}^{c}\left(
           1+ \left(\Omega_{v,w}^{c}\right)^{-1}w\right)\\
           &= v\left(v^{-1}\right)^c\lambda_{\left(\left(v^{-1}\right)^c\right)^{-1}}\left(w\right)
           \left(1 + \rho_{w}\left(\left(\left(v^{-1}\right)^c\right)^{-1}\right)\right)\\
           &= v\left(v^{-1}\right)^c\left(\left(v^{-1}\right)^c\right)^{-1}\left(1 + w\right)&\text{by \eqref{eq:condsolution}}\\
           &= v\left(1 + w\right)\\
           &= v\left(1^{c} + w\right).&\text{by $2.$}
      \end{align*}
      Therefore, the claim follows.
      \end{proof}

\medskip

\begin{rem}\label{rem:suff-cond-canc}
    Let us observe that, if $S$ and $T$ are left cancellative left semi-brace, then conditions $1.$ and $2.$ of the previous theorem are satisfied. Indeed, if $a\in S$ and $u\in T$, since $1$ is a left identity, 
    $\left(u^{1}\right)^a = u^{1+a} = u^a$. 
    Moreover, since $1^a$ is still an idempotent of $\left(T,+\right)$, it follows that $1^a + u = u = 1 + u$, hence the condition 2. is satisfied.
\end{rem}

\medskip

\begin{rem}\label{rem:altracondizione}
The condition 1. in the \cref{th:weak-asym-groups} can be replaced with the more general one
\begin{align}\label{1.'}\tag{$1'.$}
  u^{ab}
          = \left(u^{a}\right)^{\lambda_a\left(b\right)},
\end{align}
for all $a,b \in S$ and $u \in T$.
Indeed, if $a\in S$, since 
$1 + a = 1\left(1 + a\right) 
= 1 +\lambda_1\left(a\right)$,
we obtain that
\begin{align*}
    \left(u^1\right)^a
    = u^{1 + a}
    = u^{1 +\lambda_1\left(a\right)}
    =\left(u^1\right)^{\lambda_1\left(a\right)}
    \underset{(1'.)}{=} u^{1\cdot a}
    = u^{a},
\end{align*}
hence 1. is satisfied. We also underline that the condition $1'.$ holds also for left cancellative left semi-braces. Indeed, in this case, if $a,b\in S$ we have that 
$ab = a\left(1+b\right) = a + \lambda_a\left(b\right)$, and so
\begin{align*}
    u^{a b}
    = u ^{a + \lambda_a\left(b\right)}
    = \left(u^a\right)^{\lambda_a\left(b\right)},
\end{align*}
for every $u\in T$.
\end{rem}

\medskip

In the following, we construct two left inverse semi-braces in which the additive structure is a rectangular band.
\begin{ex}
Let $S$ be the left semi-brace with $(S,+)$ the left zero semigroup and $(S, \cdot)$ a group with identity $1$, and $T$ the brace with additive group $(\mathbb{Z},+)$ and multiplication $\cdot$ given by $u\cdot v:=u+(-1)^u\,v$, for all $u, v \in \mathbb{Z}$ (see \cite[Example 2]{CCoSt16}), where it is used the juxtaposition to denote the usual multiplication in $\mathbb{Z}$. Thus, the maps $r_S\left(a,b\right)=\left(1,ab\right)$ and $r_T(u,v)=\left(\left(-1\right)^{u}v, \,\left(-1\right)^{v}u\right)$ are the solutions associated to $S$ and $T$, respectively. Moreover, let $\delta$ be the anti-homomorphism from $(S,+)$ into the endomorphism semigroup  of $(T,+)$ given by $\delta(a)(u)=0$, for all $a \in S$ and $u \in T$. Note that with these assumptions the conditions 1. and 2. of \cref{th:weak-asym-groups} are trivially satisfied. It is a routine computation to check that the condition \eqref{eq:semibrace-sigma-delta} holds independently by the choice of the map $\sigma$. Now, we consider two maps $\sigma$ to obtain two distinct semi-braces which are double semidirect products of $S$ and $T$.
\begin{itemize}
    \item[-] If $\sigma(u)=\id_S$, for every $u \in T$, we get that $S \times T$ is the double semidirect product of $S$ and $T$ via $\sigma$ and $\delta$ endowed with
\begin{align*}
    (a,u)+(b,v)=(a,v) \qquad (a,u)(b,v)=(ab,\, u+(-1)^u \, v),
\end{align*}
for all $(a,u), (b, v) \in S \times T$, which we denote by $B$. By \cref{th:weak-asym-groups}, the map $r_B$ associated to $B$ given by
\begin{align*}
    r_B\left(\left(a, u \right),\left(b, v \right)\right)=\left(\left(1, u+(-1)^u \, v) \right),\left(ab, 0 \right) \right)
\end{align*}
is a solution. In addition, $r_B$ is idempotent.  
\item[-] Consider $\sigma(u)= \id_S$, if $u$ is even, instead 
$\sigma(u)= \iota$, if $u$ is odd, where $\iota:\mathbb{Z}\to \mathbb{Z}$ is the map given by $\iota\left(u\right) = -u$, for every $u \in \mathbb{Z}$.
Then, $S \times T$ is the double semidirect product of $S$ and $T$ via $\sigma$ and $\delta$ endowed with
\begin{align*}
    (a,u)+(b,v)=(a, v) 
    \qquad 
    (a,u)(b,v)=(a ^{u}b,\, u+(-1)^u \, v),
\end{align*}
for all $(a,u), (b, v) \in S \times T$, which we denote by $B$. By \cref{th:weak-asym-groups}, the map $r_B$ associated to $B$ given by
\begin{align*}
    r_B\left(\left(a, u \right),\left(b, v \right)\right)=\left(\left(1, u+(-1)^u \, v) \right),\left(^{v^{-1}}\left(^{u^{-1}}a \ b\right), 0 \right) \right),
\end{align*}
is a solution. In addition, $r_B$ is idempotent. 
 \end{itemize}
We remark that the two solutions above lie in the class of solutions associated to left semi-braces having $\rho$ as an anti-homomorphism.
\end{ex}

\medskip

The following is an example of completely simple left semi-brace. In particular, by \cite[Theorem 3]{CCoSt20x-2}, the additive structure $\left(S\times T, +\right)$ is a rectangular group.
\begin{ex}
    Let $S$ be the left semi-brace with additive group $(\mathbb{Z},+)$ and multiplication $\cdot$ given by $a \cdot b= a+(-1)^ab$, for all $a, b \in \mathbb{Z}$, where it is used the juxtaposition to denote the usual multiplication in $\mathbb{Z}$. Let $T$ be the left semi-brace with multiplicative group $(\Sym_3, \circ)$, where $\circ$ is the usual composition of maps, and addition $+$ given by $u +v:=v \,\circ\, g(v^{-1}) \, \circ \, u$, for all $u, v \in \Sym_3$, with $g$ the idempotent endomorphism of $\Sym_3$ defined by 
    \begin{align*}
    g(u)
    = \begin{cases}
      (1\, 2)&\quad \text{if $u$ is odd}\\
      \ \id_{\Sym_3} &\quad \text{otherwise}  
     \end{cases},
    \end{align*}
    for every $u\in\Sym_3$ (cf. \cite[Example 5-2]{CCoSt20x-2}). 
    If $s:\Sym_3\to \mathbb{N}_0$ is the map defined by
    $s(u) = 1$, if $u$ is an odd permutation, and $s(u) = 0$, otherwise, we can briefly write  
    $g(u) = \ (1\, 2)^{\ s\left(u\right)}$, for every $u\in\Sym_3$.
    Now, set $\delta(a)=g$, for every $a \in \mathbb{Z}$, since $g$ is an idempotent endomorphism, we have that $\delta$ in an anti-homomorphism from $(\mathbb{Z},+)$ into the endomorphism monoid of $(\Sym_3, +)$. Moreover, if $a,b \in \mathbb{Z}$ and $u,v \in \Sym_3$,  we get
    \begin{align*}
        \left(u\circ v\right)^{\lambda_{a}\left({}^ub\right)}
		&+ u\circ \left(\left(u^{-1}\right)^b
		+ w\right)=g\left(u \circ v\right)+u \circ \left(g \left(u^{-1}\right)+w\right)\\
		&=g\left(u \circ v\right)+u \circ w \circ g\left(w^{-1}\right) \circ g \left(u^{-1}\right)\\
		&=u \circ w \circ g\left(w^{-1}\right) \circ g \left(u^{-1}\right) \circ g\left(u \circ w \circ g\left(w^{-1}\right) \circ g \left(u^{-1}\right)
		\right)^{-1} \circ  g\left(u \circ v\right)\\
		&=u \circ w  \circ g\left( w^{-1} \right) \circ  g\left( v\right)=u\circ  \left(v^b +w\right),
    \end{align*}
    hence \eqref{eq:semibrace-sigma-delta} holds, independently by the choice of the map $\sigma$.
    Just to give an example, let us define $\sigma:\Sym_3\to \Aut\left(S\right)$ the homomorphism from $\left(\Sym_3, \circ\right)$ into the automorphism group of the left semi-brace $S$ given by
    \begin{align*}
    \sigma(u)
    = \begin{cases}
      \iota&\quad \text{if $u$ is odd}\\
      \ \id_{\mathbb{Z}} &\quad \text{otherwise} 
     \end{cases}
      \ = \ \iota^{\ s\left(u\right)},
    \end{align*}
    for every $u \in \Sym_3$, where $\iota:\mathbb{Z}\to \mathbb{Z}$ is the map defined by $\iota(a)=-a$, for every $a\in \mathbb{Z}$. Therefore, we obtain a double semidirect product on $S\times T$ via $\sigma$ and $\delta$ which we denote by $B$.
    Now, if we consider the maps $r_S$ and $r_T$ associated to $S$ and $T$ are respectively given by $r_S(a,b)=\left(\left(-1\right)^ab, \, \left(-1\right)^ba\right)$
    and $r_T(u,v)=\left(u\circ v\circ \ (1\, 2)^{\ s\left(v\right)}\circ u^{-1},
	u\circ \ (1\, 2)^{\ s\left(v\right)}\right)$, they are solutions and
   the conditions 1. and 2. of \cref{th:weak-asym-groups} are trivially satisfied. Hence, the map $r_B$ associated to the left semi-brace $B$ having components
    \begin{align*}
        \lambda_{\left(a,u\right)}\left(b,v\right)
        &= \left(\left(-1\right)^{ a+s\left(u\right)}\,b
        , \ 
        u\circ v\circ \, (1\, 2)^{s\left(u \circ v\right)} 
        \right)\\
        \rho_{\left(b,v\right)}\left(a,u\right)&= \left(^{ \ (1\, 2)^{s\left(u \circ v\right)} \circ v^{-1}}{\left((-1)^{b+s\left(u\right)}{a}\right)}  ,\, \ (1\, 2)^{\ s\left(u \circ v\right)} \right),
    \end{align*}
  for all $(a,u), (b,v) \in S \times T$, is a solution. Furthermore, one can check that the map $\rho$ is an anti-homomorphism from $\left(B, \cdot\right)$ into the monoid of the maps from $B$ into itself, hence, $r_B$ is a solution on $B$ and, by \cite[Proposition 2.14]{JeAr19}, $\left(B,+, \cdot\right)$ is a completely simple left semi-brace.
\end{ex}

\bigskip

Starting from the semidirect product of the inverse semigroups $S$ and $T$ via $\sigma$ in \cref{ex:semidirect2}, we show how to obtain five examples of double semidirect products by choosing all the possible maps $\delta$. 
\begin{ex}
 Let $S$ and $T$ be the inverse semigroups in \cref{ex:semidirect2}, set $a+b=b\cdot 1=1$, for all $a,b \in S$, and $u+v=uv$, for all $u, v \in T$, then $S$ and $T$ are the left inverse semi-braces in \cref{ex:prod-idempot} and \cref{ex:first-examples}, respectively. Recall that the solution $r_S$ associated to $S$ is given by $r_S(a,b)=(1,1)$ (see \cref{ex-soluzione-semibrace-be}) and, since $u^{-1}=u$, for every $u \in T$, the solution $r_T$ associated to $T$ is given by $r_T(u,v)=(u^2v, uv^2)$ (see \cref{ex-sol-semibrace-clifford}).
Now, recall that if $\tau$ is the automorphism of $S$ given by $\tau := (x\, y)$, then the map $\sigma:T\to \Aut(S)$ given by $\sigma(1) = \sigma(x) = \id_S$, and $\sigma(y) = \tau$
is a homomorphism from $(T, \cdot)$ into the automorphism group of the left inverse semi-brace $S$.
One can check that there exist five endomorphisms $\varphi$ of $(T,+)$ and, in particular, they are all idempotent. Moreover, note that to obtain $\delta$ an anti-homomorphism we have to choose 
$\delta(a)=\varphi$, for every $a \in S$, with $\varphi \in \End(T,+)$. 
In each of these cases, it is easy to prove that condition \eqref{eq:semibrace-sigma-delta} holds. Thus, for every fixed $\delta$, we have that $B:=S \times T$ is the double semidirect product of $S$ and $T$ via $\sigma$ and $\delta$ with respect to
\begin{align*}
    (a,u)+(b,v)=(1, \varphi(u)v) \qquad (a,u)(b,v)=\left(a ^u b, \, uv\right),
\end{align*}
for all $(a,u), (b,v) \in S     \times T$. Now, distinguishing the various cases, we analyze that the maps $r_B$ associated to each semi-brace $B$.
\begin{itemize}
     \item[-] If $\varphi = \id_T$:
     \begin{align*}
         r_B\left(\left(a,u\right),\left(b,v\right)\right)
         = \left(\left(1, u^{2}v\right), \ 
         \left(1, uv^{2}\right)\right),
     \end{align*}
     which clearly is a solution, since it is the semidirect product of $r_S$ and $r_T$. In particular, by \cite[Proposition 10]{CaCoSt19}, it is a cubic solution.
    \item[-] If $\varphi = k_1$, the constant map from $T$ into itself of value $1$:
    \begin{align*}
         r_B\left(\left(a,u\right),\left(b,v\right)\right)
         = \left(\left(1, uv\right), \ 
         \left(1, v^{2}\right)\right),
     \end{align*}
     which is an idempotent solution.
     \item[-] If $\varphi = k_x$, the constant map from $T$ into itself of value $x$:
     \begin{align*}
         r_B\left(\left(a,u\right),\left(b,v\right)\right)
         = \left(\left(1, xuv\right), \ 
         \left(1, 
         x v^{2}\right)\right)
     \end{align*}
     which is an idempotent solution.
     \item[-] If $\varphi$ is the map from $T$ into itself defined by $\varphi\left(1\right) = 1$ and $\varphi\left(x\right) = \varphi\left(y\right) = x$:
     \begin{align*}
         r_B\left(\left(a,u\right),\left(b,v\right)\right)
         = \left(\left(1, \varphi\left(u\right)u v\right), \ 
         \left(1, \varphi\left(u\right)v^{2}\right)\right)
     \end{align*}
     is a solution and it satisfies $r_B^5=r_B^3$.
     \item[-] If $\varphi$ the map from $T$ into itself defined by $\varphi\left(1\right) = \varphi\left(x\right) = x$ and $\varphi\left(y\right) = y$:
     \begin{align*}
         r_B\left(\left(a,u\right),\left(b,v\right)\right)
         = \left(\left(1, \varphi\left(u\right)u v\right), \ 
         \left(1, \varphi\left(u\right)v^{2}\right)\right)
     \end{align*}
     is not a solution. Indeed, if $a,b,c$ are arbitrary elements of $S$, $u = w = 1$, and $v = y$, one can check that the braid relation is not satisfied.
 \end{itemize}
\end{ex}
\medskip

\noindent In light of the previous example, it arises the following question.
\begin{que}
Let $S,T$ be left inverse semi-braces having solutions $r_S$ and $r_T$ and $B$ the double semidirect product of $S$ and $T$ via $\sigma$ and $\delta$. 
Under which assumptions the map $r_{B}$ is a solution?
\end{que}

\bigskip

\section{The asymmetric product of left inverse semi-braces}
This section aims to introduce a
generalization of the asymmetric product of left cancellative left semi-braces given in \cite{CaCoSt17}, involving left inverse semi-braces.
We highlight that, in general, this construction does not include the double semidirect product. Moreover, we provide sufficient conditions to obtain solutions.

\medskip

To present the asymmetric product of left inverse semi-braces, we need the notion of cocycle on semigroups. In particular, the following definition is inspired to that used for groups in the context of Schreier’s extension (see in \cite[Theorem 15.1.1]{Ha59}). 

\begin{defin}\label{def:delta-cocycle}
    Let $(S,+)$ and $(T,+)$ be two semigroups (not necessarily commutative) and $\delta:S\to \End\left(T\right)$ a map  
    from $S$ into the endomorphism semigroup of $T$. Set $u^{a}:= \delta\left(a\right)\left(u\right) $, for all $a\in S$ and $u\in T$, a map $\mathfrak{b}:S\times S\to T$ is called a \emph{$\delta$-cocycle} if
    \begin{align}
       \label{eq: delta-cociclo} \mathfrak{b}\left(a + b,\, c\right) + \mathfrak{b}\left(a, b\right)^c + \left(u^{b}\right)^{c} +v^c
        = \mathfrak{b}\left(a, b + c\right) + u^{b+c} + \mathfrak{b}\left(b,c\right)+v^c
    \end{align}
    holds, for all $a,b,c \in S$ and $u,v \in T$.
\end{defin}

\medskip

The notion of cocycle was already recovered in \cite[p. 173]{CaCoSt17} 
for left cancellative left semi-braces.  Specifically, the concept of $\delta$-cocycle in \cref{def:delta-cocycle} involves entirely the additive structures of $S$ and $T$, hence it is not a simple readjustment of that introduced in \cite{CaCoSt17}. Let us compare into detail the two definitions.

\begin{rem}\label{rem-cocicli}
Let $\left(S,+\right)$, $\left(T, +\right)$ be right groups, $S:= H + E$ and $T:= N + F$ where $H\cap E = \{1_S\}$, $N\cap F = \{1_T\}$, $E$ and $F$ are the sets of idempotents of $S$ and $T$, respectively, and $H = S + 1$ and $N = T + 1$.
Let $\alpha: H\to \Aut\left(N\right)$ and $\mathfrak{c}:H\times H\to N$ be maps such that $\left(\alpha, \mathfrak{c}\right)$ is a cocycle as in \cite[p. 11]{CaCoSt17}. 
Define the map $\mathfrak{b}:S\times S\to T$ given by
\begin{align*}
    \mathfrak{b}(h_1 + e_1, h_2 + e_2):= \mathfrak{c}(h_1, h_2),
\end{align*}
for every $(h_1 + e_1, h_2 + e_2)\in S\times S$, with $h_1, h_2\in H$ and $e_1 , e_2\in E$, and $\delta: S\to \End\left(T\right)$ the map given by
\begin{align*}
\delta\left(h + e\right)\left(n + f\right)= \left(n + f\right)^{h + e}
:= n^{h} + f,
\end{align*}
for all $h+e \in S$ and $n+f\in T$. If $(a,u), (b,v) \in S \times T$, with $(a,u)=(h_1+e_1, n_1+f_1)$ and $(b,v)=(h_2+e_2, n_2+f_2)$, since $e_1$ and $f_1$ are idempotents, then we have that
\begin{align*}
\left(a,u\right)+\left(b,v\right)
&=\left(h_1+e_1+h_2+e_2,\, \mathfrak{b}\left(h_1+e_1, h_2+e_2\right)+\left(n_1+f_1\right)^{\left(h_2+f_2\right)}+n_2+f_2\right)\\
&=\left(h_1+h_2+e_2,\, \mathfrak{c}\left(h_1,h_2\right)+n_1^{h_2}+f_1+n_2+f_2\right)\\
&= \left(h_1+h_2+e_2,\,  \mathfrak{c}\,\left(h_1,h_2\right)+n_1^{h_2}+n_2+f_2\right),
\end{align*}
which is exactly the sum in \cite[Theorem 12]{CaCoSt17}. Now, recalling conditions $1.$ and $2.$ in \cite[p. 173]{CaCoSt17}, namely,
\begin{align*}
    &\left(n^{h_1}\right)^{h_2}
    = -\mathfrak{c}\left(h_1,h_2\right)
    + n^{h_1 + h_2} + \mathfrak{c}\left(h_1,h_2\right)\\
    &\mathfrak{c}\left(h_1+h_2, h_3\right)
    + \mathfrak{c}\left(h_1, h_2\right)^{h_3}
     =\mathfrak{c}\left(h_1, h_2 + h_3\right) + \mathfrak{c}\left(h_2,h_3\right),
\end{align*}
for all $h_1,h_2,h_3\in H$ and $n\in N$, 
if $a= h_1 + e_1$, $b= h_2 + e_2$, $c = h_3 + e_3$ are elements of $S$, and $u = n_1 + f_1$, $v = n_2 + f_2$ are elements of $T$, 
it follows that  
\begin{align*}
        &\mathfrak{b}\left(a + b,\, c\right) + \mathfrak{b}\left(a, b\right)^c + \left(u^{b}\right)^{c} + v^c\\
        &=\mathfrak{c}\left(h_1 + h_2, h_3\right)
        + \mathfrak{c}\left(h_1, h_2\right)^{h_3}
        + \left(n_1^{h_2}\right)^{h_3} + f_1 +v^c\\
        &= \mathfrak{c}\left(h_1 + h_2, h_3\right)
        + \mathfrak{c}\left(h_1, h_2\right)^{h_3}
        + \left(n^{h_2}\right)^{h_3} +v^c\\
        &=\mathfrak{c}\left(h_1, h_2+ h_3\right) + \mathfrak{c}\left(h_2, h_3\right)+ \left(n^{h_2}\right)^{h_3} +v^c\\
        &=\mathfrak{c}\left(h_1, h_2+h_3\right)+ n^{h_2+h_3}+f_1+\mathfrak{c}\left(h_2,h_3\right)+v^c  \\
        &= \mathfrak{b}\left(a, b + c\right) + u^{b+c} +\mathfrak{b}\left(b,c\right)+v^c,
    \end{align*}
i.e., $\mathfrak{b}$ is a $\delta$-cocycle.
\end{rem}

\medskip

In the following theorem we provide the construction of the asymmetric product of left inverse semi-braces.
\begin{theor}\label{th:asymmetric-inv-semi}
		Let $S$,$T$ be left inverse semi-braces,  $\sigma: T \to \Aut\left(S\right)$ a homomorphism from $\left(T,\cdot\right)$  into the automorphism group of the left inverse semi-brace $S$, with $^u a:= \sigma(u)(a)$, for all $a \in S$ and $u \in T$, and $\delta:S\to \End\left(T\right)$ a map from $S$ into the endomorphism semigroup of $\left(T, +\right)$. If  $\mathfrak{b}$ is a $\delta$-cocycle such that
	\begin{align}\label{eq:cocycle-semibrace}
\mathfrak{b}\left(a \, ^ub, \lambda_{a}\left({}^uc\right)\right)
		+ \left(uv\right)^{\lambda_{a}\left({}^uc\right)}
		+ u\left(\mathfrak{b}\left(^{u^{-1}}\left(a^{-1}\right),c\right)+\left(u^{-1}\right)^c\right)=u\left(\mathfrak{b}\left(b,c\right) + v^c\right)
	\end{align}
		holds, for all $a,b,c \in S$ and $u,v\in T$, then $S \times T$ with respect to
		\begin{align*}
		 \left(a,u\right)+\left(b,v\right)&:=\left(a+b, \, \mathfrak{b}\left(a, b \right)+u^b+v\right)\\
		 \left(a,u\right)\left(b,v\right)&:=\left(a \ ^{u}{b}, \, uv\right),
		\end{align*}
		for all $\left(a,u\right), \left(b,v\right) \in S \times T$,
		is a inverse semi-brace. We call such a left inverse semi-brace the \emph{asymmetric product} of $S$ and $T$ via $\sigma$, $\delta$, and $\mathfrak{b}$.
		\begin{proof}
		Initially, by applying \eqref{eq: delta-cociclo}, it is a routine computation to see that $(S \times T, +)$ is a semigroup. Moreover, by \cref{th:char-semiprod}, $(S \times T, \cdot)$ is an inverse semigroup. Now, to get the claim we have only to check that \eqref{key1} holds. If $(a,u), (b,v), (c, w) \in S \times T$, we obtain
			 \begin{align*}
	 \left(a,u\right)\left(\left(b,v\right)+\left(c,w\right)\right)=\left( a \ {}^{u}\left(b+c\right), u \left(\mathfrak{b}\left(b,c\right)+v^c+w\right)\right)
	\end{align*}
	and
		\begin{align*}
		    &\left(a,u\right)\left(b,v\right)
		    + \left(a,u\right)\left(\left(a,u\right)^{-1} + \left(c,w\right)\right)\\
		    &= \left(a \,{}^{u}b, \, uv\right)
		    + \left(a,u\right)
		    \left({}^{u^{-1}}a^{-1} + c, \,
		    \mathfrak{b}\left({}^{u^{-1}}a^{-1}, c\right) +\left(u^{-1}\right)^{c} + w\right)\\
		    &= \left(a \,{}^{u}b, \, uv\right)
		    + \left(a\ {}^{u}\left({}^{u^{-1}}a^{-1} + c\right), 
		    u\left(\mathfrak{b}\left({}^{u^{-1}}a^{-1}, c\right) +\left(u^{-1}\right)^{c} + w  \right)\right)\\
		    &= \left(a ^u b +a \left(^{uu^{-1}}{a^{-1}}+ ^uc\right), 
		    \, \mathfrak{b}\left(a {}^{u}b, a \left(^{uu^{-1}}{a^{-1}}+ ^uc\right)\right) + \left(uv\right)^{a \left(^{uu^{-1}}{a^{-1}}+ ^uc\right)}\right.\\
		    &\qquad\left.+\, u\left(\mathfrak{b}\left(^{u^{-1}}\left(a^{-1}\right),c\right)+\left(u^{-1}\right)^c+w\right)\right)\\
		    &= \left( a ^u b +\lambda_a \left(^uc\right), \mathfrak{b}\left(a {}^{u}b, \lambda_a \left(^uc\right)\right) + \left(uv\right)^{\lambda_a \left(^uc\right)}+ \,u\, \left(\mathfrak{b}\left(^{u^{-1}}\left(a^{-1}\right),c\right)+\left(u^{-1}\right)^c+w\right)\right),
	 \end{align*}
	 where in the last equality we apply \eqref{aut-idemp}. Now, we observe that
	\begin{align*}
	    a \, {}^{u} b +\lambda_a \left({}^{u}c\right)
	    = a\left({}^{u} b+ {}^{u} c\right)
	    = a\, {}^{u}\left(b + c\right)
	\end{align*}
	and
	\begin{align*}
	\mathfrak{b}&\left(a \,{}^{u}b, \lambda_a \left(^uc\right)\right) + \left(uv\right)^{\lambda_a \left(^uc\right)}+ \,u\, \left(\mathfrak{b}\left(^{u^{-1}}\left(a^{-1}\right),c\right)+\left(u^{-1}\right)^c+w\right)\\
	&=  \mathfrak{b}\left(a \,{}^{u}b, \lambda_a \left(^uc\right)\right) + \left(uv\right)^{\lambda_a \left(^uc\right)}+ \,u\, \left(\mathfrak{b}\left(^{u^{-1}}\left(a^{-1}\right),c\right)+\left(u^{-1}\right)^c\right) +u\left(u^{-1}+w\right) \\
	&=u \left(\mathfrak{b}\left(b,c\right)+v^c\right)+u\left(u^{-1}+w\right) &\mbox{by \eqref{eq:cocycle-semibrace}}\\
	&=u \left(\mathfrak{b}\left(b,c\right)+v^c+w\right).
	\end{align*}
Therefore, $S \times T$	is a left inverse semi-brace.
		\end{proof}
\end{theor}		

\medskip

Let us compare \cite[Theorem 12]{CaCoSt17} with \cref{th:asymmetric-inv-semi} in the specific case of two left cancellative left semi-braces.
\begin{rem}
Let $\left(S,+\right)$, $\left(T, +\right)$ be right groups, $S:= H + E$ and $T:= N + F$ where $H\cap E = \{1\}$, $N\cap F = \{1\}$, $E$ and $F$ are the sets of idempotents of $S$ and $T$, respectively, and $H = S + 1$ and $N = T + 1$. Let us consider the maps $\mathfrak{b}$ and $\delta$ defined in \cref{rem-cocicli} and $a= h_1 + e_1$, $b= h_2 + e_2$, $c = h_3 + e_3 \in S$, and  $u = n_1 + f_1$, $v = n_2 + f_2 \in T$. Recalling that in any left cancellative left semi-brace it holds $x \cdot y=x +\lambda_x(y)$, by \cite[Proposition 7-1.]{CaCoSt17}, we obtain that
\begin{align*}
    a \, ^ub
    &= h_1 + e_1 +\lambda_{h_1 + e_1} \left(^{\left(n_1+f_1\right)}\left(h_2 + e_2\right)\right) 
    = h_1 + e_1 + \lambda_{h_1 + e_1} \left(^{\left(n_1+f_1\right)}\left(h_2\right)\right)
    + \lambda_{h_1 + e_1} \left(^{\left(n_1+f_1\right)}\left(e_2\right)\right)
    \\
    &= \underbrace{\left(h_1 + e_1\right)\left(^{\left(n_1+f_1\right)}\left(h_2\right)\right) + 1}_{H} + \underbrace{\lambda_{h_1+e_1}\left(^{\left(n_1+f_1\right)}e_2\right)}_{E}
\end{align*}
and
\begin{align*}
    \lambda_{a}\left({}^uc\right)
    = \underbrace{\lambda_{h_1 + e_1}\left(^{\left(n_1+f_1\right)}h_3 \right) +1}_{H} +\underbrace{ \lambda_{h_1 + e_1}\left(e_3\right)}_{E}
\end{align*}
and, by \cite[Lemma 11-1.]{CaCoSt17},
\begin{align*}
    a^{-1}
    = \left(h_1 + e_1\right)^{-1} e_1 + \left(h_1 + e_1\right)^{-1} h_1
    = \underbrace{\rho_{e_1}\left(h_1^{-1}\right)}_{H} + \underbrace{\left(\lambda_{h_1^{-1}}\left(e_1\right)\right)^{-1}}_{E}.
\end{align*}
Analogously, 
$u^{-1} = \underbrace{\rho_{f_1}\left(n_1^{-1}\right)}_{N} + \underbrace{\left(\lambda_{n_1^{-1}}\left(f_1\right)\right)^{-1}}_{F}$.
Moreover, $uv
    = \underbrace{\left(n_1 + f_1\right) n_2 + 1}_{N}+ \underbrace{\lambda_{n_1 + f_1}\left(f_2\right)}_{F}$, hence we get
\begin{align*}
    \left(u v\right)^{\lambda_a\left(^{u}{c}\right)}
    = \underbrace{\left(\left(n_1 + f_1\right) n_2 + 1\right)^{\lambda_{h_1+e_1}\left(^{\left(n_1+f_1\right)}h_3+1\right)}}_{N} + \underbrace{\lambda_{n_1 + f_1}\left(f_2\right)}_{F}.
\end{align*}
Since $\lambda_{n_1 + f_1}\left(f_2\right)\in F$, by condition $(5)$ of \cite[Theorem 12]{CaCoSt17} and by recalling that
$
    \mathfrak{b}(h_1 + e_1, h_2 + e_2)= \mathfrak{c}(h_1, h_2),
$
for every $(h_1 + e_1, h_2 + e_2)\in S\times S$,
it follows that
\begin{align*}
     &\mathfrak{b}\left(a \, ^ub, \lambda_{a}\left({}^uc\right)\right)
		+ \left(uv\right)^{\lambda_{a}\left({}^uc\right)}
		+ u\left(\mathfrak{b}\left(^{u^{-1}}\left(a^{-1}\right),c\right)+\left(u^{-1}\right)^c\right)\\
		&= \mathfrak{c}\left(\left(h_1 + e_1\right)\,^{n_1+f_1}h_2,\, \lambda_{h_1 + e_1}\left(^{\left(n_1+f_1\right)}h_3 + 1 \right) +1\right)\\
		&\quad+ \left(\left(n_1 + f_1\right) n_2 + 1\right)^{\lambda_{h_1+e_1}\left(^{\left(n_1+f_1\right)}h_3 + 1\right)} + \lambda_{n_1 + f_1}\left(f_2\right)\\
		&\quad+ \left(n_1 + f_1\right)\left(
		\mathfrak{c}\left(^{u^{-1}}\rho_{e_1}\left(h_1^{-1}\right),h_3\right)
		+ \rho_{f_1}\left(n_1^{-1}\right)^{h_3}\right)\\
		&=  \mathfrak{c}\left(\left(h_1 + e_1\right)\, ^{n_1+f_1}h_2  + 1, \,\lambda_{h_1 + e_1}\left(^{n_1+f_1}h_3\right) +1\right)\\
		&\quad+ \left(\left(n_1 + f_1\right) n_2 + 1\right)^{\lambda_{h_1+e_1}\left(^{\left(n_1+f_1\right)}h_3+1\right)}\\
		&\quad+ \left(n_1 + f_1\right)\left(
		\mathfrak{c}\left(^{\left(n_1+f_1\right)^{-1}}\rho_{e_1}\left(h_1^{-1}\right),h_3\right)
		+ \rho_{f_1}\left(n_1^{-1}\right)^{h_3}\right)\\
		&= \left(n_1 + f_1\right) \left(\mathfrak{c}\left(h_2,h_3\right) + n_2^{h_3}\right)\\
		& =
       u\left(\mathfrak{b}\left(b,c\right) + v^c\right),
   \end{align*}
   namely, condition \eqref{eq:cocycle-semibrace} is satisfied.
\end{rem}

\medskip

In the following theorem we give the map associated to any asymmetric product of left inverse semi-braces. We omit the proof since it is similar to that in \cref{r_weakproduct}. We recall the notation already adopted, that is $\Omega_{u,v}^{a}:= \left(u^{-1}\right)^a+v$,
for all $a\in S$ and $u,v\in T$. 

\begin{prop}
    Let $S$, $T$ be left inverse semi-braces and $B$ the asymmetric product of $S$ and $T$ via $\sigma$ and $\delta$ and $\mathfrak{b}$. Then, the map $r_B$ associated to $B$ is given by
    \begin{align*}
        r_B\left(\left(a,u\right), \left(b,v\right)\right)
        =
        &\left(\left(\lambda_a\left(^{u}{b}\right), u\left(\mathfrak{b}\left(^{u^{-1}}a^{-1},b\right) +\Omega_{u,v}^{b}\right)\right),\right.\\
        &\left.
        \left(^{\mathfrak{b}\left(^{u^{-1}}a^{-1},b\right) + \left(\Omega_{u,v}^{b}\right)^{-1}u^{-1}}{\rho_{^u b}\left(a\right)}, \,  \left(\mathfrak{b}\left(^{u^{-1}}a^{-1},b\right)+ \Omega_{u,v}^{b}\right)^{-1}v\right)\right),
    \end{align*}
    for all $\left(a,u\right), \left(b,v\right)\in S\times T$.
\end{prop}
\medskip

\noindent Note that the second component of $\rho_{(b,v)}(a,u)$ can be written also as
\begin{align*}
\left(\mathfrak{b}\left(^{u^{-1}}a^{-1},b\right)+ \Omega_{u,v}^{b}\right)^{-1}v 
&= \rho_v\left(\left(\mathfrak{b}\left(^{u^{-1}}a^{-1},b\right)+ \left(u^{-1}\right)^b\right)^{-1}\right),
\end{align*}
for all $(a,u), (b,v) \in S \times T$. 
\medskip

Now, we focus on sufficient conditions that allow for constructing new solutions through the asymmetric product of left semi-braces.

\begin{theor}\label{th:asym-groups}
    Let $S$, $T$ be left semi-braces and 
    $B$ the asymmetric product of $S$ and $T$ via $\sigma$, $\delta$, and $\mathfrak{b}$. If $r_S$ and $r_T$ are solutions associated to $S$ and $T$, respectively, and the following are satisfied  
      \begin{enumerate}
          \item $\left(u^{1}\right)^a = u^a$,
          \item $\mathfrak{b}\left(1,a\right)+u=1 + u$,
          \item $\mathfrak{b}\left(a,1+b\right)=\mathfrak{b}\left(a,b\right)$,
      \end{enumerate}
      for all $a,b\in S$ and $u\in T$, then the map $r_B$ associated to $B$ is a solution.
      \begin{proof}
      To prove that the map $r_B$ is a solution, we check that the condition \eqref{eq:condsolution}  in \cref{th-generalized-sol} is satisfied.\\ 
      If $(a,u), (b, v), (c,w) \in S \times T$, by $2.$, we get
      \begin{align*}
     &\left(a,u\right)+\left(b,v\right)\left(\left(1,1\right)+\left(c,w\right)\right)\\
        &=\left(a+b\left(1+^v c\right),\, \mathfrak{b}\left(a, \, b\left(1+^v c\right) \right)+u^{b\left(1+\, ^vc\right)}+v\left(\mathfrak{b}\left(1, c \right)+1^c+w\right)\right)\\
        &=\left(a+b\left(1+^v c\right),\, \mathfrak{b}\left(a, \, b\left(1+^v c\right) \right)+u^{b\left(1+\, ^vc\right)} + v\left(1 + w\right)\right),
      \end{align*}
      since $1^c$ is idempotent and by \cref{prop-middleunit-clifford-preston} it is a middle unit. Moreover,    
\begin{align*}
    &\left(a,u\right)
           +\lambda_{\left(b,v\right)}\left(c,w\right)
           \left(\left(1,1\right) + \rho_{\left(c,w\right)}\left(b,v\right)\right)\\
           &= \left(a + \lambda_b\left(^v c\right)\left(1 + \rho_{{}^v c}\left(b\right)\right),\, \mathfrak{b}\left(a,\lambda_b\left(^v c\right)\left(1 + \rho_{{}^v c}\left(b\right)\right) \right)+u^{\lambda_b\left(^v c\right)\left(1 + \rho_{{}^v c}\left(b\right)\right)} \right.\\
         &\left.\qquad + v\left(\mathfrak{b}\left( ^ {v^{-1}} {b^{-1}},\,c\right)+\Omega_{v,w}^{c}\right)\left(\mathfrak{b}\left(1,\, ^{\left(\mathfrak{b}\left( ^ {v^{-1}} {b^{-1}},\,c\right)+\Omega_{v,w}^{c}\right)^{-1}v^{-1}}{\rho_{^{v}{c}}{\left(b\right)}}\right)\right.\right.\\
         &\left.\left. \qquad
         +1^{^{\left(\mathfrak{b}\left( ^ {v^{-1}} {b^{-1}},\,c\right)+\Omega_{v,w}^{c}\right)^{-1}v^{-1}}{\rho_{^{v}{c}}{\left(b\right)}}}
+ \left(\mathfrak{b}\left( ^ {v^{-1}} {b^{-1}},\,c\right)+\Omega_{v,w}^{c}\right)^{-1}w\right) \right).
\end{align*}
Note that the first components are equal since $r_S$ is a solution. In addition,
   \begin{align*}
       \mathfrak{b}\left(a,\lambda_b\left(^v c\right)\left(1 + \rho_{{}^v c}\left(b\right)\right) \right)&=\mathfrak{b}\left(a,1+\lambda_b\left(^v c\right)\left(1 + \rho_{{}^v c}\left(b\right)\right) \right) &\mbox{by $3.$}\\
       &=\mathfrak{b}\left(a, \, b\left(1+^v c\right)\right) &\mbox{by \eqref{eq:condsolution}}
   \end{align*}
   Furthermore, since $1^{^{\left(\mathfrak{b}\left( ^ {v^{-1}} {b^{-1}},\,c\right)+\Omega_{v,w}^{c}\right)^{-1}v^{-1}}{\rho_{^{v}{c}}{\left(b\right)}}}$ is an idempotent it is a middle unit and so, by $2.$, we get
   \begin{align*}
      &\mathfrak{b}\left(1,\, ^{\left(\mathfrak{b}\left( ^ {v^{-1}} {b^{-1}},\,c\right)+\Omega_{v,w}^{c}\right)^{-1}v^{-1}}{\rho_{^{v}{c}}{\left(b\right)}}\right)
         +\left(\mathfrak{b}\left( ^ {v^{-1}} {b^{-1}},\,c\right)+\Omega_{v,w}^{c}\right)^{-1}w\\
         &=1+\left(\mathfrak{b}\left( ^ {v^{-1}} {b^{-1}},\,c\right)+\Omega_{v,w}^{c}\right)^{-1}w
   \end{align*}
   Moreover, by $1.$ and \eqref{eq:condsolution}, it follows that $u^{\lambda_b\left(^v c\right)\left(1 + \rho_{{}^v c}\left(b\right)\right)}=u^{b\left(1+\, ^vc\right)}$, hence,
   \begin{align*}
       &u^{\lambda_b\left(^v c\right)\left(1 + \rho_{{}^v c}\left(b\right)\right)}+v\left(\mathfrak{b}\left( ^ {v^{-1}} {b^{-1}},\,c\right)+\Omega_{v,w}^{c}\right)\left(1+\left(\mathfrak{b}\left( ^ {v^{-1}} {b^{-1}},\,c\right)+\Omega_{v,w}^{c}\right)^{-1}w\right)\\
       &=u^{b\left(1+\, ^vc\right)}+v\left(\mathfrak{b}\left( ^ {v^{-1}} {b^{-1}},\,c\right)+\left(v^{-1}\right)^{c}\right)\cdot\\
       &\qquad\cdot\lambda_{\left(\mathfrak{b}\left( ^ {v^{-1}} {b^{-1}},\,c\right)+\left(v^{-1}\right)^{c}\right)^{-1}}\left(w\right)\left(1+\rho_w\left(\left(\mathfrak{b}\left( ^ {v^{-1}} {b^{-1}},\,c\right)+\left(v^{-1}\right)^{c}\right)^{-1}\right)\right)\\
       &=u^{b\left(1+\, ^vc\right)}+v\left(\mathfrak{b}\left( ^ {v^{-1}} {b^{-1}},\,c\right)+\left(v^{-1}\right)^{c}\right)\left(\mathfrak{b}\left( ^ {v^{-1}} {b^{-1}},\,c\right)+\left(v^{-1}\right)^{c}\right)^{-1}\left(1+w\right) &\mbox{by \eqref{eq:condsolution}}\\
       &=u^{b\left(1+\, ^vc\right)}+v\left(1+w\right).
   \end{align*}
      Thus, the second components are equal. Therefore, the claim follows.
      \end{proof}
\end{theor}

\medskip

\begin{rem}
    As observed in \cref{rem:suff-cond-canc}, if $S$ and $T$ are left cancellative left semi-braces, then conditions $1.$ and $3.$ of the previous theorem are satisfied. 
    In addition, note that if $a\in S$ and $u\in T$, since  $\left(1,1\right)$ is a left identity in $S\times T$, we have
    \begin{align*}
        \left(a, 1 + u\right) 
        = \left(a,u\right) = 
        \left(1,1\right) + \left(a,u\right)
        = \left(1 + a, \, \mathfrak{b}\left(1,a\right) + 1^a+ u\right)
        =\left(a,    \mathfrak{b}\left(1,a\right) + u\right),
    \end{align*}
    thus the condition $2.$ also holds. 
\end{rem}

\medskip

\begin{rem}
    Note that, similarly to \cref{rem:altracondizione}, assuming that conditions $2.$ and $3.$ in the \cref{th:asym-groups} are satisfied, instead of condition $1.$ we may assume the more general one
    \begin{align}\label{1.''}\tag{$1'$.}
        u^{ab}
        = \left(u^{a}\right)^{\lambda_a\left(b\right)},
    \end{align}
for all $a,b \in S$ and $u \in T$.
\end{rem}

\medskip

\begin{ex}
Let $X$ be a set and $\cdot$ a binary operation on $X$ such that $\left(X, \cdot \right)$ is an abelian group. Thus, setting $a+b:= bf\left(b^{-1}a\right)$, for all $a,b\in X$, with $f$ an idempotent endomorphism of 
$\left(X, \cdot \right)$, we have that $S:= \left(X, +, \cdot\right)$ is a left semi-brace, see \cite[Example 5, case 3]{CCoSt20x-2}.
Moreover, let $T := \left(X, +, \cdot\right)$ be the left brace with the sum given by $u + v = u v$, for all $u,v\in X$.\\
Let us consider the map $\delta: S\to \End\left(T\right)$ such that $u^a = 1$, for all $a\in S$ and $u\in T$ and $\mathfrak{b}:S\times S\to T$ the $\delta$-cocycle defined by $\mathfrak{b}\left(a,b\right) = a$, for all $a,b\in S$. Thus, if $\sigma: T\to \Aut\left(S\right)$ is the map given by $^{u}a = a$, for all $a\in S$ and $u\in T$, it follows that \eqref{eq:cocycle-semibrace} is trivially satisfied. Therefore, $B:=S \times T$ is the asymmetric product of $S$ and $T$ via $\sigma, \delta$, and $\mathfrak{b}$ endowed with
\begin{align*}
    \left(a,u\right)+\left(b,v\right)=\left(a+b,a+v\right) \quad \text{and} \quad \left(a,u\right)\left(b,v\right)= \left(ab,uv\right),
\end{align*}
for all $ \left(a,u\right),\left(b,v\right)\in S \times T$. Now, the solutions $r_S$ and $r_T$ are given by $r_S\left(a,b\right)=\left(abf\left(b^{-1}a^{-1}\right),f\left(ab\right)\right)$ and $r_T\left(u,v\right)=\left(v, v^{-1}uv\right)$, respectively. Since the conditions in \cref{th:asym-groups} trivially hold, then we obtain that the map
\begin{align*}
    r_B\left(\left(a,u\right),\left(b,v\right)\right)=\left( \left( abf\left(b^{-1}a^{-1}\right), ua^{-1}v\right),\left(f\left(ab\right),a\right)\right)
\end{align*}
is a solution. Furthermore, it holds $r_B^3=r_B^2$.
\end{ex}

\medskip

The following is a simple class of examples of asymmetric product of inverse semi-braces. 
\begin{ex}
Let $S$ be an arbitrary left inverse semi-brace and $T$ the left inverse semi-brace with $\left(T,+\right)$ the left zero semigroup and $(T,\cdot)$ an upper semilattice with join $1$. Moreover, let $\sigma:T\to \Aut\left(S\right)$ be the homomorphism from $\left(T, \cdot\right)$ into the automorphism group of the left inverse semi-brace $S$ given by  ${}^{u}a = a$, for all $a\in S$ and $u\in T$, and $\delta:S\to \End(T)$ an arbitrary map from $S$ into the endomorphism monoid of $\left(T, +\right)$, and $\mathfrak{b}: S\times S\to T$ the $\delta$-cocycle defined $\mathfrak{b}\left(a,b\right) = 1$, for all $a, b \in S$. Therefore, by \cref{th:asymmetric-inv-semi}, $S\times T$ is a left inverse semi-brace with
\begin{align*}
		 \left(a,u\right)+\left(b,v\right)
		 =\left(a + b, \, 1\right)
		 \qquad
		 \left(a,u\right)\left(b,v\right) =\left(ab, \, uv\right),
		\end{align*}
		for all $\left(a,u\right), \left(b,v\right) \in S \times T$. Moreover, the map $r_B$ is given by
\begin{align*}
    r_B\left(\left(a,u\right), \left(b,v\right)\right)
    = \left(\left(\lambda_a\left(b\right), 1\right),\,
    \left(\rho_b\left(a\right), 1\right)\right)
\end{align*}
which is trivially a solution.
\end{ex}

\medskip

Finally, sufficient conditions to obtain that the map $r_B$ associated to an asymmetric product $B$ of two left inverse semi-braces $S$ and $T$ are not yet known to the authors. Thus, it arises the following question.
\begin{que}
Let $S,T$ be left inverse semi-braces having solutions $r_S$ and $r_T$ and $B$ the asymmetric product of $S$ and $T$ via $\sigma, \delta$, and a $\delta$-cocycle $\mathfrak{b}$. Under which assumptions is the map $r_{B}$ a solution?
\end{que}

\bibliographystyle{elsart-num-sort}  
\bibliography{bibliography}

\def\cprime{$'$}
\begin{thebibliography}{10}
\expandafter\ifx\csname url\endcsname\relax
  \def\url#1{\texttt{#1}}\fi
\expandafter\ifx\csname urlprefix\endcsname\relax\def\urlprefix{URL }\fi

\bibitem{AcBo20}
E.~Acri, M.~Bonatto, Skew braces of size {$pq$}, Comm. Algebra 48~(5) (2020)
  1872--1881.
\newline\urlprefix\url{https://doi.org/10.1080/00927872.2019.1709480}

\bibitem{AcLuVe20}
E.~Acri, R.~Lutowski, L.~Vendramin, Retractability of solutions to the
  {Y}ang-{B}axter equation and {$p$}-nilpotency of skew braces, Internat. J.
  Algebra Comput. 30~(1) (2020) 91--115.
\newline\urlprefix\url{https://doi.org/10.1142/S0218196719500656}

\bibitem{BaCeJeOk19}
D.~Bachiller, F.~Ced\'{o}, E.~Jespers, J.~Okni\'{n}ski, Asymmetric product of
  left braces and simplicity; new solutions of the {Y}ang-{B}axter equation,
  Commun. Contemp. Math. 21~(8) (2019) 1850042, 30.
\newline\urlprefix\url{https://doi.org/10.1142/S0219199718500426}

\bibitem{Ba72}
R.~J. Baxter, Partition function of the eight-vertex lattice model, Ann.
  Physics 70 (1972) 193--228.
\newline\urlprefix\url{https://doi.org/10.1016/0003-4916(72)90335-1}

\bibitem{CaCaDC20}
E.~Campedel, A.~Caranti, I.~Del~Corso, Hopf-{G}alois structures on extensions
  of degree {$p^2q$} and skew braces of order {$p^2 q$}: the cyclic {S}ylow
  {$p$}-subgroup case, J. Algebra 556 (2020) 1165--1210.
\newline\urlprefix\url{https://doi.org/10.1016/j.jalgebra.2020.04.009}

\bibitem{CaCaMiPi19x}
M.~Castelli, F.~Catino, M.~M. Miccoli, G.~Pinto, Dynamical extensions of
  quasi-linear left cycle sets and the {Y}ang-{B}axter equation, J. Algebra
  Appl. 18~(11) (2019) 1950220, 16.
\newline\urlprefix\url{https://doi.org/10.1142/s0219498819502207}

\bibitem{CaCaPi18}
M.~Castelli, F.~Catino, G.~Pinto, A new family of set-theoretic solutions of
  the {Y}ang-{B}axter equation, Comm. Algebra 46~(4) (2018) 1622--1629.
\newline\urlprefix\url{https://doi.org/10.1080/00927872.2017.1350700}

\bibitem{CaCaPi19}
M.~Castelli, F.~Catino, G.~Pinto, Indecomposable involutive set-theoretic
  solutions of the {Y}ang--{B}axter equation, J. Pure Appl. Algebra 223~(10)
  (2019) 4477--4493.
\newline\urlprefix\url{https://doi.org/10.1016/j.jpaa.2019.01.017}

\bibitem{CaCaPi20}
M.~Castelli, F.~Catino, G.~Pinto, About a question of {G}ateva-{I}vanova and
  {C}ameron on square-free set-theoretic solutions of the {Y}ang-{B}axter
  equation, Comm. Algebra 48~(6) (2020) 2369--2381.
\newline\urlprefix\url{https://doi.org/10.1080/00927872.2020.1713328}

\bibitem{CaCaSt20x}
M.~Castelli, F.~Catino, P.~Stefanelli, Left non-degenerate set-theoretic
  solutions of the {Y}ang-{B}axter equation and dynamical extensions of q-cycle
  sets, arXiv preprint arXiv:2001.10774.
\newline\urlprefix\url{https://arxiv.org/abs/2001.10774}

\bibitem{CaPiRu20}
M.~Castelli, G.~Pinto, W.~Rump, On the indecomposable involutive set-theoretic
  solutions of the {Y}ang-{B}axter equation of prime-power size, Comm. Algebra
  48~(5) (2020) 1941--1955.
\newline\urlprefix\url{https://doi.org/10.1080/00927872.2019.1710163}

\bibitem{CCoSt16}
F.~Catino, I.~Colazzo, P.~Stefanelli, Regular subgroups of the affine group and
  asymmetric product of radical braces, J. Algebra 455 (2016) 164--182.
\newline\urlprefix\url{http://dx.doi.org/10.1016/j.jalgebra.2016.01.038}

\bibitem{CaCoSt17}
F.~Catino, I.~Colazzo, P.~Stefanelli, Semi-braces and the {Y}ang-{B}axter
  equation, J. Algebra 483 (2017) 163--187.
\newline\urlprefix\url{https://doi.org/10.1016/j.jalgebra.2017.03.035}

\bibitem{CCoSt19}
F.~Catino, I.~Colazzo, P.~Stefanelli, Skew left braces with non-trivial
  annihilator, J. Algebra Appl. 18~(2) (2019) 1950033, 23.
\newline\urlprefix\url{https://doi.org/10.1142/S0219498819500336}

\bibitem{CCoSt20}
F.~Catino, I.~Colazzo, P.~Stefanelli, The matched product of set-theoretical
  solutions of the {Y}ang-{B}axter equation, J. Pure Appl. Algebra 224~(3)
  (2020) 1173--1194.
\newline\urlprefix\url{https://doi.org/10.1016/j.jpaa.2019.07.012}

\bibitem{CaCoSt19}
F.~Catino, I.~Colazzo, P.~Stefanelli, The {M}atched {P}roduct of the
  {S}olutions to the {Y}ang--{B}axter {E}quation of {F}inite {O}rder, Mediterr.
  J. Math. 17, 58 (2020).
\newline\urlprefix\url{https://doi.org/10.1007/s00009-020-1483-y}

\bibitem{CCoSt20x-2}
F.~Catino, I.~Colazzo, P.~Stefanelli, Set-theoretic solutions to the
  {Y}ang-{B}axter equation and generalized semi-braces, arxiv preprint
  arXiv:2004.01606.
\newline\urlprefix\url{https://arxiv.org/abs/2004.01606}

\bibitem{CaMaSt20}
F.~Catino, M.~Mazzotta, P.~Stefanelli, Set-theoretical solutions of the
  {Y}ang–{B}axter and pentagon equations on semigroups, Semigroup Forum
  100~(3) (2020) 1--26.
\newline\urlprefix\url{https://doi.org/10.1007/s00233-020-10100-x}

\bibitem{Ce18}
F.~Ced\'{o}, Left braces: solutions of the {Y}ang-{B}axter equation, Adv. Group
  Theory Appl. 5 (2018) 33--90.
\newline\urlprefix\url{https://doi.org/10.4399/97888255161422}

\bibitem{CeJeOk20-2}
F.~Ced\'{o}, E.~Jespers, J.~Okni\'{n}ski, An abundance of simple left braces
  with abelian multiplicative {S}ylow subgroups, Rev. Mat. Iberoam.
\newline\urlprefix\url{https://www.ems-ph.org/journals/of_article.php?jrn=RMI&doi=1168}

\bibitem{CeJeOk14}
F.~Ced\'{o}, E.~Jespers, J.~Okni\'{n}ski, Braces and the {Y}ang-{B}axter
  equation, Comm. Math. Phys. 327~(1) (2014) 101--116.
\newline\urlprefix\url{https://doi.org/10.1007/s00220-014-1935-y}

\bibitem{CeJeOk21}
F.~Ced\'{o}, E.~Jespers, J.~Okni\'{n}ski, Every finite abelian group is a
  subgroup of the additive group of a finite simple left brace, J. Pure Appl.
  Algebra 225~(1) (2021) 106476.
\newline\urlprefix\url{http://www.sciencedirect.com/science/article/pii/S0022404920301778}

\bibitem{CeJeVe20x}
F.~Ced\'o, E.~Jespers, C.~Verwimp, Structure monoids of set-theoretic solutions
  of the {Y}ang-{B}axter equation, arxiv preprint arXiv:1912.09710.
\newline\urlprefix\url{https://arxiv.org/abs/1912.09710}

\bibitem{CeSmVe19}
F.~Ced\'{o}, A.~Smoktunowicz, L.~Vendramin, Skew left braces of nilpotent type,
  Proc. Lond. Math. Soc. (3) 118~(6) (2019) 1367--1392.
\newline\urlprefix\url{https://doi.org/10.1112/plms.12209}

\bibitem{ClPr61}
A.~H. Clifford, G.~B. Preston, The algebraic theory of semigroups. {V}ol. {I},
  Mathematical Surveys, No. 7, American Mathematical Society, Providence, R.I.,
  1961.

\bibitem{Charl19}
K.~Cvetko-Vah, C.~Verwimp, Skew lattices and set-theoretic solutions of the
  {Y}ang-{B}axter equation, J. Algebra 542 (2020) 65--92.
\newline\urlprefix\url{https://doi.org/10.1016/j.jalgebra.2019.10.007}

\bibitem{Dr92}
V.~G. Drinfel\cprime~d, On some unsolved problems in quantum group theory, in:
  Quantum groups ({L}eningrad, 1990), vol. 1510 of Lecture Notes in Math.,
  Springer, Berlin, 1992, pp. 1--8.
\newline\urlprefix\url{https://doi.org/10.1007/BFb0101175}

\bibitem{ESS99}
P.~Etingof, T.~Schedler, A.~Soloviev, Set-theoretical solutions to the quantum
  {Y}ang-{B}axter equation, Duke Math. J. 100~(2) (1999) 169--209.
\newline\urlprefix\url{http://dx.doi.org/10.1215/S0012-7094-99-10007-X}

\bibitem{GaMa08}
T.~Gateva-Ivanova, S.~Majid, Matched pairs approach to set theoretic solutions
  of the {Y}ang-{B}axter equation, J. Algebra 319~(4) (2008) 1462--1529.
\newline\urlprefix\url{https://doi.org/10.1016/j.jalgebra.2007.10.035}

\bibitem{GaVa98}
T.~Gateva-Ivanova, M.~Van~den Bergh, Semigroups of {$I$}-type, J. Algebra
  206~(1) (1998) 97--112.
\newline\urlprefix\url{http://dx.doi.org/10.1006/jabr.1997.7399}

\bibitem{GuVe17}
L.~Guarnieri, L.~Vendramin, Skew braces and the {Y}ang-{B}axter equation, Math.
  Comp. 86~(307) (2017) 2519--2534.
\newline\urlprefix\url{https://doi.org/10.1090/mcom/3161}

\bibitem{Ha59}
M.~Hall, Jr., The theory of groups, The Macmillan Co., New York, N.Y., 1959.

\bibitem{Ho95}
J.~M. Howie, Fundamentals of semigroup theory, vol.~12 of London Mathematical
  Society Monographs. New Series, The Clarendon Press, Oxford University Press,
  New York, 1995, oxford Science Publications.

\bibitem{JeKuVaVe19}
E.~Jespers, L.~. Kubat, A.~Van~Antwerpen, L.~Vendramin, Factorizations of skew
  braces, Math. Ann. 375~(3-4) (2019) 1649--1663.
\newline\urlprefix\url{https://doi.org/10.1007/s00208-019-01909-1}

\bibitem{JeAr19}
E.~Jespers, A.~Van~Antwerpen, Left semi-braces and solutions of the
  {Y}ang-{B}axter equation, Forum Math. 31~(1) (2019) 241--263.
\newline\urlprefix\url{https://doi.org/10.1515/forum-2018-0059}

\bibitem{Ku83}
M.~Kunze, Zappa products, Acta Math. Hungar. 41~(3-4) (1983) 225--239.
\newline\urlprefix\url{https://doi.org/10.1007/BF01961311}

\bibitem{Le17}
V.~Lebed, Cohomology of idempotent braidings with applications to factorizable
  monoids, Internat. J. Algebra Comput. 27~(4) (2017) 421--454.
\newline\urlprefix\url{https://doi.org/10.1142/S0218196717500229}

\bibitem{LuYZ00}
J.-H. Lu, M.~Yan, Y.-C. Zhu, On the set-theoretical {Y}ang-{B}axter equation,
  Duke Math. J. 104~(1) (2000) 1--18.
\newline\urlprefix\url{http://dx.doi.org/10.1215/S0012-7094-00-10411-5}

\bibitem{MaMa19}
M.~E. Malandro, Enumeration of finite inverse semigroups, Semigroup Forum
  99~(3) (2019) 679--723.
\newline\urlprefix\url{https://doi.org/10.1007/s00233-019-10054-9}

\bibitem{MaSh18}
D.~K. Matsumoto, K.~Shimizu, Quiver-theoretical approach to dynamical
  {Y}ang-{B}axter maps, J. Algebra 507 (2018) 47--80.
\newline\urlprefix\url{https://doi.org/10.1016/j.jalgebra.2018.04.003}

\bibitem{Ze19}
K.~Nejabati~Zenouz, Skew braces and {H}opf-{G}alois structures of {H}eisenberg
  type, J. Algebra 524 (2019) 187--225.
\newline\urlprefix\url{https://doi.org/10.1016/j.jalgebra.2019.01.012}

\bibitem{Ni83}
W.~R. Nico, On the regularity of semidirect products, J. Algebra 80~(1) (1983)
  29--36.
\newline\urlprefix\url{https://doi.org/10.1016/0021-8693(83)90015-7}

\bibitem{Pe84}
M.~Petrich, Inverse semigroups, Pure and Applied Mathematics (New York), John
  Wiley \& Sons, Inc., New York, 1984, a Wiley-Interscience Publication.

\bibitem{PeRe99}
M.~Petrich, N.~R. Reilly, Completely regular semigroups, vol.~23 of Canadian
  Mathematical Society Series of Monographs and Advanced Texts, John Wiley \&
  Sons, Inc., New York, 1999, a Wiley-Interscience Publication.

\bibitem{Pr86}
G.~B. Preston, Semidirect products of semigroups, Proc. Roy. Soc. Edinburgh
  Sect. A 102~(1-2) (1986) 91--102.
\newline\urlprefix\url{https://doi.org/10.1017/S0308210500014505}

\bibitem{Ru05}
W.~Rump, A decomposition theorem for square-free unitary solutions of the
  quantum {Y}ang-{B}axter equation, Adv. Math. 193~(1) (2005) 40--55.
\newline\urlprefix\url{https://doi.org/10.1016/j.aim.2004.03.019}

\bibitem{Ru07}
W.~Rump, Braces, radical rings, and the quantum {Y}ang-{B}axter equation, J.
  Algebra 307~(1) (2007) 153--170.
\newline\urlprefix\url{https://doi.org/10.1016/j.jalgebra.2006.03.040}

\bibitem{Ru16}
W.~Rump, Quasi-linear cycle sets and the retraction problem for set-theoretic
  solutions of the quantum {Y}ang-{B}axter equation, Algebra Colloq. 23~(1)
  (2016) 149--166.
\newline\urlprefix\url{https://doi.org/10.1142/S1005386716000183}

\bibitem{Ru19-1}
W.~Rump, Set-theoretic solutions to the {Y}ang-{B}axter equation, skew-braces,
  and related near-rings, J. Algebra Appl. 18~(8) (2019) 1950145, 22.
\newline\urlprefix\url{https://doi.org/10.1142/S0219498819501457}

\bibitem{Ru20-3}
W.~Rump, Affine structures of decomposable solvable groups, J. Algebra 556
  (2020) 725--749.
\newline\urlprefix\url{https://doi.org/10.1016/j.jalgebra.2020.04.004}

\bibitem{Ru20-2}
W.~Rump, Classification of indecomposable involutive set-theoretic solutions to
  the {Y}ang--{B}axter equation, Forum Math. 32~(4) (2020) 891--903.
\newline\urlprefix\url{https://doi.org/10.1515/forum-2019-0274}

\bibitem{SmVe18}
A.~Smoktunowicz, L.~Vendramin, On skew braces (with an appendix by {N}. {B}yott
  and {L}. {V}endramin), J. Comb. Algebra 2~(1) (2018) 47--86.
\newline\urlprefix\url{https://doi.org/10.4171/JCA/2-1-3}

\bibitem{So00}
A.~Soloviev, Non-unitary set-theoretical solutions to the quantum
  {Y}ang-{B}axter equation, Math. Res. Lett. 7~(5-6) (2000) 577--596.
\newline\urlprefix\url{http://dx.doi.org/10.4310/MRL.2000.v7.n5.a4}

\bibitem{StVo20x}
D.~Stanovsk{\`y}, P.~Vojt{\v{e}}chovsk{\`y}, Idempotent solutions of the
  {Y}ang-{B}axter equation and twisted group division, arXiv preprint
  arXiv:2002.02854.
\newline\urlprefix\url{https://arxiv.org/abs/2002.02854}

\bibitem{Ve16}
L.~Vendramin, Extensions of set-theoretic solutions of the {Y}ang-{B}axter
  equation and a conjecture of {G}ateva-{I}vanova, J. Pure Appl. Algebra
  220~(5) (2016) 2064--2076.
\newline\urlprefix\url{https://doi.org/10.1016/j.jpaa.2015.10.018}

\bibitem{Wa15}
S.~Wazzan, Zappa-{S}z{\'e}p products of semigroups, Applied Mathematics 6~(06)
  (2015) 1047.
\newline\urlprefix\url{http://dx.doi.org/10.4236/am.2015.66096}

\bibitem{Ya67}
C.~N. Yang, Some exact results for the many-body problem in one dimension with
  repulsive delta-function interaction, Phys. Rev. Lett. 19 (1967) 1312--1315.
\newline\urlprefix\url{https://doi.org/10.1103/PhysRevLett.19.1312}

\bibitem{Ya16}
D.~Yang, The interplay between {$k$}-graphs and the {Y}ang-{B}axter equation,
  J. Algebra 451 (2016) 494--525.
\newline\urlprefix\url{https://doi.org/10.1016/j.jalgebra.2016.01.001}

\end{thebibliography}

\end{document}